\documentclass{amsart}

\usepackage[dvipdfmx]{graphicx}
\usepackage{amsmath}
\usepackage{amssymb}
\usepackage{amsthm}

\usepackage{ygtab}

\title[Free-fermions and skew stable Grothendieck polynomials]{Free-fermions and skew stable Grothendieck polynomials}
\author{Shinsuke Iwao}
\address{Department of Mathematics, Tokai University, 4-1-1, Kitakaname, Hiratsuka, Kanagawa 259-1292, Japan.}
\email{iwao@tokai.ac.jp}
\date{\today}

\newtheorem{thm}{Theorem}[section]
\newtheorem{prop}[thm]{Proposition}
\newtheorem{lemma}[thm]{Lemma}
\newtheorem{defi}[thm]{Definition}
\newtheorem{example}[thm]{Example}
\newtheorem{rem}{Remark}[section]

\newtheorem{cor}[thm]{Corollary}

\def\CC{\mathord{\mathbb{C}}}

\def\ZZ{\mathord{\mathbb{Z}}}
\def\QQ{\mathord{\mathbb{Q}}}

\newcommand{\bm}[1]{\mbox{\boldmath{$#1$}}}
\def\slsl{/\!\!/} 

\def\ket#1{\vert #1 \rangle}
\def\bra#1{\langle #1 \vert}

\def\zet#1{\left\vert{#1}\right\vert}
\def\ee{\mathbf{e}}

\def\Ygtab#1{{\ygtab{#1}}}

\begin{document}

\begin{abstract}

Skew stable Grothendieck polynomials are $K$-theoretic analogues of skew Schur polynomials.
We give a free-fermionic presentation of skew stable Grothendieck polynomials and their dual symmetric functions. 
By using our presentation, we derive a family of determinantal formulas, which are $K$-analogues of the Jacobi-Trudi formula for skew Schur functions.
We also introduce a combinatorial method to calculate certain expansions of skew (dual) stable Grothendieck polynomials by using the non-commutative supersymmetric Schur functions.


\smallskip
\noindent \textbf{MSC Classes.} 05E05, 05A10
\end{abstract}
\maketitle

\section{Introduction}\label{sec:intro}

Symmetric Grothendieck polynomials are $K$-theoretic analogues of Schur symmetric polynomials, which represent a Schubert class in the $K$-theory of the Grassmann variety~\cite{lascoux1982structure,lascoux1983symmetry}.
Although the Grothendieck polynomials are not stable when the number of independent variables tends to infinity, it is possible to define a certain ``symmetric function" by taking an appropriate limit, a \textit{stable limit}~\cite{fomin1994grothendieck}.
The symmetric function which is obtained is called the \textit{stable Grothendieck polynomial}. 

For a positive integer $n$, and a partition $\lambda$ of length $\ell(\lambda)\leq n$, the ($\beta$-)\textit{Grothendieck polynomial} $G_\lambda(x_1,\dots,x_n)$ corresponding to $\lambda$ is a symmetric polynomial that is expressed as follows~\cite{IKEDA201322,kirillov2016some,yeliussizov2017duality}: 
\begin{equation}\label{eq:ratio}
G_\lambda(x_1,\dots,x_n)=
\frac{\det\left(
x_i^{\lambda_j+n-j}(1+\beta x_i)^{j-1}
\right)_{1\leq i,j\leq n}
}{
\prod_{1\leq i<j\leq n}(x_i-x_j)
}.
\end{equation}
For example, $G_{(1)}(x_1,\dots,x_n)=\sum_{i=1}^{n} \beta^{i-1} e_i(x_1,\dots,x_n)$, where $e_i(x_1,\dots,x_n)$ is the $i$-th elementary symmetric polynomial in $n$ variables.
Taking the stable limit $n\to \infty$, one obtains the infinite sum $G_{(1)}(x)=e_1(x)+\beta e_2(x)+\beta^2e_3(x)+\cdots$.
Here $e_i(x)=e_i(x_1,x_2,\dots)$ is the $i$-th elementary symmetric function in infinitely many variables.
We note the fact that a stable Grothendieck polynomial is not a polynomial in general but an infinite series of symmetric functions.
It reduces to a symmetric polynomial in $n$ variables when we substitute $x_{n+1}=x_{n+2}=\cdots=0$.

In \cite{iwao2020freefermion}, the author of this paper introduced a new characterization of the stable Grothendieck polynomials in terms of the \textit{boson-fermion correspondence}~ \cite{date1983method,miwa2012solitons} (see also \cite{alexandrov2013free}).
It was shown that straightforward calculations based on this characterization lead some determinantal formulas for the stable Grothendieck polynomials.

Our aim in this paper is to generalize the previous result to ($\beta$-)skew stable Grothendieck polynomials~\cite{buch2002littlewood} and their dual symmetric functions~\cite{lam2007combinatorial}.
The goal of this work is to present:
\begin{itemize}
\item a free-fermionic presentation of two types of skew stable Grothendieck polynomials $G_{\lambda\slsl\mu}(x)$ and $G_{\lambda/\mu}(x)$ (Theorems \ref{thm:G//}, \ref{thm:G/}),
\item a free-fermionic presentation of skew dual stable Grothendieck polynomials $g_{\lambda/\mu}(x)$ (Theorem \ref{thm:g/}),
\item determinantal formulas for the symmetric functions $G_{\lambda\slsl \mu}(x)$, $G_{\lambda/\mu}(x)$, $g_{\lambda/\mu}(x)$, $s_\mu^\perp G_\lambda(x)$, $s_\mu^\perp g_\lambda(x)$, etc. (Propositions \ref{prop:det_G//}, \ref{prop:det_G/}, \ref{prop:det_G//_finite}, \ref{prop:det_G/_finite}, \ref{prop:det_g} and Remark \ref{rem:various}),
\item a combinatorial method to calculate the $G_{\lambda/\mu}$-expansion of $s_\nu G_{\lambda/\mu}(x)$, and the $g_{\lambda/\mu}$-expansion of $s_\nu g_{\lambda/\mu}(x)$ (Sections \ref{sec:G-expan} and \ref{sec:g-expan}).
\end{itemize}

\subsection{Relation to Previous research}

A skew stable Grothendieck polynomial $G_{\lambda/\mu} $ is known as a representative of a Schubert variety corresponding to a 321-avoiding permutation in the $K$-theory of the flag variety~\cite{buch2002littlewood,fomin1994grothendieck}.
Their combinatorial presentation was studied in Buch's paper \cite{buch2002littlewood}, in which an explicit formula for $G_{\lambda/\mu}$ was given in terms of \textit{skew set-valued tableaux}.
He also introduced the symmetric function $G_{\lambda\slsl\mu}$, which is another generalization of a skew Schur function.

Many types of determinantal formulas for skew Grothendieck polynomials and their generalizations have been proved in the context of the geometry of the flag bundle ~\cite{anderson2017k,HUDSON2017115,hudson2019segre,kim2020jacobitrudi,matsumura2019flagged}.
Hudson-Ikeda-Matsumura-Naruse~\cite{HUDSON2017115} derived a determinantal formula for the \textit{factorial Grothendieck polynomials} by studying the \textit{connective $K$-theory} of the Grassmann bundle.
Our determinantal formula for $G_{\lambda/\mu}$ (Proposition \ref{prop:det_G/}) is a special case of that given by Matsumura~\cite{matsumura2019flagged} for the \textit{flagged Grothendieck polynomials}.
These formulas are also studied in relation to $K$-theoretic Brill-Noether theorem~\cite{anderson2017k,chan2017euler}.

In the study of these symmetric functions, 
the \textit{non-commutative Schur function}~\cite{yeliussizov2017duality,YELIUSSIZOV2019453} is a very useful tool to describe their combinatorial aspects.
In recent works, Kim~\cite{kim2020jacobitrudi} and Amanov-Yeliussizov~\cite{amanov2020determinantal} have been developing original techniques to study a quite wider class of symmetric functions in a purely combinatorial way.
They independently proved a determinantal formula for $g_{\lambda/\mu}$, which is different from Proposition \ref{prop:det_g}.

It is worth noting that the Grothendieck polynomials have appeared in studies of \textit{classical and quantum integrable systems}. 
Motegi-Sakai~\cite{motegi2013vertex,motegi2014k} proved that the Grothendieck polynomials (and their variants) can be derived as wave functions of certain quantum integrable systems.
In their works, many algebraic formulas, such as determinantal formulas, were shown in the context of the quantum mechanics.
Nagai and the author of this paper~\cite{iwao2018discrete} showed that the \textit{discrete Toda equation}~\cite{hirota1977} admits a special solution that is written as a dual stable Grothendieck polynomial.
For recent developments on this topic, see Ikeda-Iwao-Maeno~\cite{ikeda2017peterson} and Motegi~\cite{motegi2018symmetric}.

\subsection{Organization of the paper}

In Section \ref{sec:free_fermions}, we give a quick review of basic facts about the boson-fermion correspondence.
Some propositions are stated without proof.
For readers who are not familiar with this topic, we recommend the review article \cite{alexandrov2013free} by Alexandrov-Zabrodin.
We also include some useful lemmas which will be used in the later sections.

In Sections \ref{sec:Gr}--\ref{sec:skewdual}, we deal with free-fermionic presentations of symmetric functions.
Section \ref{sec:Gr} is a short review of \cite{iwao2020freefermion}, a previous work of the author of this paper, in which the author gives a free-fermionic presentation of the stable Grothendieck polynomials $G_\lambda$ and their duals $g_\lambda$.
In Section \ref{sec:skew}, we extend our method to the skew stable Grothendieck polynomials $G_{\lambda\slsl \mu}$, $G_{\lambda/\mu}$, and give a free-fermionic presentation of them.
By virtue of this presentation, it becomes a straightforward process to obtain their determinantal formulas. 
We prove in Section \ref{sec:skewdual} that similar results are also valid for the skew dual Grothendieck polynomials $g_{\lambda/\mu}$.

In Sections \ref{sec:noncommSchur}--\ref{sec:expansion_of_s^perp}, we introduce a combinatorial method for calculating a $G_{\lambda/\mu}$-expansion and a $g_{\lambda/\mu}$-expansion of Schur polynomials.
Our method relies on the theory of \textit{non-commutative Schur functions}.
In Section \ref{sec:noncommSchur}, we define the non-commutative Schur functions $s_\lambda(\bm{u}_r)$, $s_\lambda(\bm{v}_r)$, \textit{etc}.~ and calculate their actions on the Fock space $\mathcal{F}$.
In the following sections, we construct a new algorithm to obtain a $G_{\lambda/\mu}$-expansion of $s_\nu G_{\lambda/\mu}$(Section \ref{sec:G-expan}) and a $g_{\lambda/\mu}$-expansion of $s_\nu g_{\lambda/\mu}$ (Section \ref{sec:g-expan}). 
We also consider the symmetric functions of the form $s^\perp_\nu G_{\lambda/\mu}$, $s^\perp_\nu g_{\lambda/\mu}$ in Section \ref{sec:expansion_of_s^perp}.

\section{boson-fermion correspondence}\label{sec:free_fermions}

In this section, we give a brief review of basic facts about boson-fermion correspondence.
Throughout the paper, we write $[A,B]=AB-BA$ and  $[A,B]_+=AB+BA$.

\subsection{The module of free fermions $\mathcal{A}$ and the Fock space $\mathcal{F}$}\label{sec:prelim}

For a field $k$ of characteristic $0$, we consider a $k$-algebra $\mathcal{A}$ generated by \textit{free fermions} $\psi_n$, $\psi_n^\ast$ ($n\in \ZZ$) which satisfies
\begin{equation}\label{eq:free-fermions-relation}
[\psi_m,\psi_n]_+=[\psi^\ast_m,\psi^\ast_n]_+=0,\qquad
[\psi_m,\psi^\ast_n]_+=\delta_{m,n}.
\end{equation}

Let $\ket{0}$, $\bra{0}$ denote the \textit{vacuum vectors}:
\[
\psi_m\ket{0}=\psi^\ast_n\ket{0}=0,\quad
\bra{0}\psi_n=\bra{0}\psi^\ast_m=0,\qquad m< 0,\ n\geq 0.
\]
The \textit{Fock space} (over $k$) is the $k$-space $\mathcal{F}$ generated by the vectors
\begin{equation}\label{eq:elementary-vactors}
\psi_{n_1}\psi_{n_2}\cdots \psi_{n_r}\psi^\ast_{m_1}\psi^\ast_{m_2}\cdots \psi^\ast_{m_s}\ket{0},\
(r,s\geq 0,\ n_1>\dots>n_r\geq 0>m_s>\dots>m_1).
\end{equation}
We also consider the $k$-space $\mathcal{F}^\ast$ generated by the vectors
\begin{equation}\label{eq:elementary-vetcors-dual}
\bra{0}\psi_{m_s}\cdots \psi_{m_2}\psi_{m_1}\psi^\ast_{n_r}\cdots \psi^\ast_{n_2}\psi^\ast_{n_1},\
(r,s\geq 0,\ n_1>\dots>n_r\geq 0>m_s>\dots>m_1).
\end{equation}

It is known that the vectors \eqref{eq:elementary-vactors} are linearly independent (see, for example, \cite[\S 4 and \S 5.2]{kac2013bombay}).
The vectors \eqref{eq:elementary-vetcors-dual} are also linearly independent.
Using \eqref{eq:free-fermions-relation} repeatedly, we can prove that $\mathcal{F}$ is a left $\mathcal{A}$-module, and that $\mathcal{F}^\ast$ is a right $\mathcal{A}$-module.

There exists an anti-algebra involution on $\mathcal{A}$ defined by
\[
{}^\ast:\mathcal{A}\to \mathcal{A};\quad \psi_n\leftrightarrow \psi_n^\ast.
\]
Note that $(ab)^\ast=b^\ast a^\ast$ and $(a^\ast)^\ast=a$.
Therefore we have the $k$-linear isomorphism
\[
\omega:\mathcal{F}\to {\mathcal{F}}^\ast,\quad X\ket{0}\mapsto \bra{0}{X}^\ast.
\]

\subsection{Vacuum expectation value and Wick's theorem}\label{sec:Wick}

Let  
\[
{\mathcal{F}}^\ast\otimes_k\mathcal{F}\to k,\quad 
\bra{w}\otimes \ket{v}\mapsto \langle{w}\vert v\rangle,
\]
be the \textit{vacuum expectation value} \cite[\S 4.5]{miwa2012solitons}, that is, the unique $k$-bilinear map that satisfies (i) $\langle 0\vert 0\rangle=1$,
(ii) $(\bra{w}\psi_n) \ket{v}=\bra{w} (\psi_n\ket{v})$,
and
(iii) $(\bra{w}\psi_n^\ast) \ket{v}=\bra{w} (\psi_n^\ast\ket{v})$.
For any expression $X$, we often use the abbreviations
$\bra{w}X\ket{v}:=(\langle{w}\vert X)\ket{v}
=\bra{w}(\vert X\ket{v})
$ and $\langle X\rangle=\bra{0}X\ket{0}$.

For an integer $m$,  we set
\[
\ket{m}=
\begin{cases}
\psi_{m-1}\psi_{m-2}\cdots \psi_0\ket{0}, & m\geq 0,\\
\psi^\ast_{m} \cdots\psi^\ast_{-2}\psi^\ast_{-1}\ket{0}, & m<0,
\end{cases}\quad
\bra{m}=
\begin{cases}
\bra{0}\psi^\ast_0\psi^\ast_1\dots \psi^\ast_{m-1}, & m\geq 0,\\
\bra{0}\psi_{-1}\psi_{-2}\dots \psi_{m}, & m<0.
\end{cases}
\]

\begin{thm}[Wick's theorem
(see {\cite[\S 2]{alexandrov2013free}, \cite[Exercise 4.2]{miwa2012solitons}})
]\label{thm:Wick}
Let $\{m_1,\dots,m_r\}$ and $\{n_1,\dots,n_{r}\}$ be sets of integers.
Then we have
\[
\langle 
\psi_{m_1}\cdots\psi_{m_{r}}
\psi^\ast_{n_r}\cdots\psi^\ast_{n_{1}}
\rangle
=\det(\langle \psi_{m_i}\psi^\ast_{n_j} \rangle)_{1\leq i,j\leq r}.
\]
\end{thm}

For $m=\{m_1,\dots,m_r\}$ and $n=\{n_1,\dots,n_s\}$ with $m_1>\dots>m_r$, $n_1>\dots>n_s$, we write
\[
\delta_{m,n}=
\begin{cases}
1,& r=s \mbox{ and }m_i=n_i \mbox{ for all $i$},\\
0,&\mbox{otherwise}.
\end{cases}
\]
\begin{cor}\label{cor:dual}
Assume that $m_1>\dots>m_{r}\geq -r$ and $n_1>\dots>n_{s}\geq -s$.
Then we have
\[
\bra{-r}
\psi^\ast_{m_r}\cdots\psi^\ast_{m_{1}}
\psi_{n_1}\cdots\psi_{n_{s}}
\ket{-s}
=\delta_{m,n},
\]
where $m=\{m_1,\dots,m_r\}$, $n=\{n_1,\dots,n_s\}$.
\end{cor}

\subsection{The boson-fermion correspondence}\label{sec:boson-fermion}

Let $:\bullet :$ be the \textit{normal ordering} (see \cite[\S 2]{alexandrov2013free}, \cite[\S 5.2]{miwa2012solitons}) of free-fermions.
We define the operator $a_m$ ($m\in \ZZ$) by $a_m=\sum_{k\in \ZZ} :\psi_k\psi^\ast_{k+m}:$ which acts on the Fock space $\mathcal{F}$.
The operators $a_m$ satisfy the following commutative relations
\begin{equation}\label{eq:relation_added}
[a_m,a_n]=m\delta_{m+n,0},\qquad
[a_m,\psi_n]=\psi_{n-m},\qquad
[a_m,\psi^\ast_n]=-\psi^\ast_{n+m}.
\end{equation}
See \cite[\S 5.3]{miwa2012solitons} for proofs.

For an infinitely many independent variables $x_1,x_2,\dots$, we let $\Lambda$ denote the $k$-algebra of symmetric functions in $x_1,x_2,\dots$.
Let $p_i(x)=x_1^i+x_2^i+\cdots\in \Lambda$ be the \textit{$i$-th power sum}.
Consider the operator
\[
H(x)=\sum_{n>0}\frac{p_n(x)}{n}a_n,
\]
which acts on the $k$-space $\Lambda\otimes_k \mathcal{F}$.
By using \eqref{eq:relation_added}, we can prove the commutative relation
\begin{equation}\label{eq:exp_H_psi}
e^{H(x)}\psi_n=\left(\sum_{i=0}^\infty h_i(x)\psi_{n-i}\right)e^{H(x)},
\end{equation}
where $h_i(x)$ is the \textit{$i$-th complete symmetric function}.

\begin{thm}[{\cite[Lemma 9.5]{miwa2012solitons}}, {\cite[Theorem 6.1]{kac2013bombay}}]\label{thm:Schur}
Let $\lambda=(\lambda_1,\lambda_2,\dots,\lambda_r)$ be a partition, that is, a weakly decreasing sequence of nonnegative integers.
Set $\ket{\lambda}:=\psi_{\lambda_1-1}\psi_{\lambda_2-2}\cdots\psi_{\lambda_r-r}\ket{-r}$.
Then, we have $s_\lambda(x)=\bra{0}e^{H(x)}\ket{\lambda}$, where $s_\lambda(x)$ is the \textit{Schur function} labeled by $\lambda$,
\end{thm}

\subsection{Hall inner product and adjoint action}

There uniquely exists a non-degenerate bilinear form $\langle \cdot,\cdot \rangle$ over $\Lambda$ that satisfies
$
\langle s_\lambda,s_\mu\rangle =\delta_{\lambda,\mu}
$,
which is called the \textit{Hall inner product} \cite[I,\S 4]{macdonald1998symmetric}.
Given $f\in \Lambda$, let $f^\perp:\Lambda\to \Lambda$ be the unique $k$-linear map that satisfies $\langle f^\perp g,h\rangle=\langle g,fh\rangle$ for any $g,h\in \Lambda$.
\begin{thm}[{\cite[Theorem 5.1]{miwa2012solitons}}]\label{thm:an-fermion-adjoint}
Let $\ket{v}\in \mathcal{F}$ and $f=\bra{0}e^{H(x)}\ket{v}$.
Then we have 
\[
p_n^\perp f=\bra{0}e^{H(x)}a_{n}\ket{v}.
\]
\end{thm}

Let $S_\lambda(P_1,P_2,\dots)$ be the polynomial in $P_1,P_2,\dots$ that satisfies
\begin{equation}\label{eq:Sato-Schur}
S_\lambda(p_1(x),p_2(x),\dots,)=s_\lambda(x).
\end{equation}
For example, we have $S_{(1)}=P_1$, $S_{(2)}=\frac{1}{2}P_1^2+\frac{1}{2}P_2$, and $S_{(1,1)}=\frac{1}{2}P_1^2-\frac{1}{2}P_2$.
Put $\bra{\lambda}:=\omega(\ket{\lambda})$.
From $S_\mu(a_{-1},a_{-2},\dots)\ket{0}=\ket{\mu}$ and $a_n^\ast=a_{-n}$, it follows that
\[
\bra{0}S_\mu(a_{1},a_{2},\dots)=\bra{\mu}.
\]
Let $s_{\lambda/\mu}(x)=s^\perp_\mu s_\lambda(x)$ be the \textit{skew Schur function} \cite[I,\S 5]{macdonald1998symmetric}.
As $[e^{H(x)},a_i]=0$ for $i>0$, we obtain
\begin{equation}\label{eq:skew-Schur-fermion}
s_{\lambda/\mu}(x)
=
\bra{0}e^{H(x)}S_\mu(a_1,a_2,\dots)\ket{\lambda}
=
\bra{\mu}e^{H(x)}\ket{\lambda}
\end{equation}
from Theorem \ref{thm:an-fermion-adjoint}.
Extending (\ref{eq:skew-Schur-fermion}) bilinearly, we have:
\begin{thm}\label{thm:adjoint_action_fermion}
For $\ket{v},\ket{w}\in\mathcal{F}$, let $f(x)=\bra{0}e^{H(x)}\ket{v}$ and $g(x)=\bra{0}e^{H(x)}\ket{w}$.
Then we have
\[
g^\perp f(x)=\bra{w}e^{H(x)}\ket{v},\qquad{\mbox{where}}\qquad \bra{w}=\omega(\ket{w}).
\]
\end{thm}
\begin{cor}\label{cor:adjoint_action}
Let $\ket{v},\ket{w}\in \mathcal{F}$ and $f=\bra{w}e^{H(x)}\ket{v}$.
Then we have 
\[
p_n^\perp f=\bra{w}e^{H(x)}a_{n}\ket{v}.
\]
\end{cor}

\subsection{Operators $e^\Theta$ and $e^\theta$}\label{sec:etheta_eTheta}

The following two operators~\cite{iwao2020freefermion} are essential for our calculations:
\[
\Theta=\beta a_{-1}-\frac{\beta^2}{2}a_{-2}+\frac{\beta^3}{3}a_{-3}-\cdots,\quad
\theta=\Theta^\ast=\beta a_{1}-\frac{\beta^2}{2}a_{2}+\frac{\beta^3}{3}a_{3}-\cdots.
\]

Let $\psi(z)=\sum_{n\in \ZZ}\psi_nz^n$ and $\psi^\ast(z)=\sum_{n\in \ZZ}\psi^\ast_nz^n$ be the generating functions of $\psi_n$ and $\psi^\ast_n$ with a formal variable $z$.
They satisfy $[a_n,\psi(z)]=z^n\psi(z)$ and $[a_n,\psi^\ast(z)]=-z^{-n}\psi^\ast(z)$.

Since
\[
e^X Ye^{-X}=
Y+\mathrm{ad}_X\cdot Y+\frac{(\mathrm{ad}_X)^2}{2!}\cdot Y+\frac{(\mathrm{ad}_X)^3}{3!}\cdot Y+\cdots
\]
($\mathrm{ad}_X\cdot Y=[X,Y]$), we have
\begin{equation}\label{eq:e_theta_action}
e^{\Theta}\psi(z)e^{-\Theta}=(1+\beta z^{-1})\psi(z),\qquad
e^{\theta}\psi(z)e^{-\theta}=(1+\beta z)\psi(z).
\end{equation}

In the rest of this section, we list several lemmas which will be used later.

\begin{lemma}\label{lemma:basic_1}
We have the following equations.
\begin{enumerate}
\gdef\theenumi{\roman{enumi}}
\item \label{item:basic_1_1}
$e^\Theta \psi_ne^{-\Theta}=
\psi_n+\beta \psi_{n+1}
$.
\item \label{item:basic_1_2}
$e^\theta \psi_ne^{-\theta}=
\psi_n+\beta \psi_{n-1}
$.
\item \label{item:basic_1_3}
$
e^{H(x)}e^{\Theta}=
\prod_{l=1}^\infty
(1+\beta x_l)
\cdot
e^{\Theta}e^{H(x)}
$.
\item \label{item:basic_1_4}
$
e^{H(x)}\psi(z)=
\left(
\sum_{i=0}^{\infty}h_i(x)z^i
\right)
\psi(z)e^{H(x)}
$.
\end{enumerate}
\end{lemma}
\begin{proof}
Every equation is shown from  (\ref{eq:exp_H_psi}--\ref{eq:e_theta_action}) by straightforward calculations.
\end{proof}

For a sequence of integers $n=(n_1,\dots,n_r)$, set
\begin{align}
&\ket{n}^G=
\ket{n_1,\dots,n_r}^G
:=\psi_{n_1-1}e^{\Theta}
\psi_{n_2-2}e^{\Theta}
\cdots
\psi_{n_r-r}e^{\Theta}\ket{-r},\label{eq:vecG}
\\
&\ket{n}^g=
\ket{n_1,\dots,n_r}^g
:=
\psi_{n_1-1}e^{-\theta}
\psi_{n_2-2}e^{-\theta}
\cdots
\psi_{n_r-r}e^{-\theta}
\ket{-r}.\label{eq:vecg}
\end{align}
\begin{rem}\label{rem:n^g_tips}
Since $e^{-\theta}\ket{-r}=\ket{-r}$, \eqref{eq:vecg} can be rewritten as
\[
\ket{n}^g=
\psi_{n_1-1}e^{-\theta}
\psi_{n_2-2}e^{-\theta}
\cdots
\psi_{n_r-r}
\ket{-r}.
\]
\end{rem}
\begin{lemma}\label{lemma:basic_2}
The vectors $(\ref{eq:vecG}$--$\ref{eq:vecg})$ satisfy the following equations.
\begin{enumerate}
\item \label{item:basic_2_1}
$\ket{n_1,\dots,n_{r-1},-1}^G=(-\beta)\ket{n_1,\dots,n_{r-1},0}^G$,
\item \label{item:basic_2_1a}
$\ket{\dots,n-1,n,\dots}^G=(-\beta)\ket{\dots,n,n,\dots}^G$,
\item \label{item:basic_2_2}
$n_r<0\Rightarrow\ket{n_1,\dots,n_{r}}^g=0$,
\item \label{item:basic_2_3}
$\ket{\dots,n-1,n,\dots}^g=(-\beta)\ket{\dots,n-1,n-1,\dots}^g$.
\end{enumerate}
\end{lemma}
\begin{proof}
Equations (\ref{item:basic_2_1}--\ref{item:basic_2_1a}) follow from Lemma \ref{lemma:basic_1} \eqref{item:basic_1_1}, and Equations (\ref{item:basic_2_2}--\ref{item:basic_2_3}) follow from Lemma \ref{lemma:basic_1} \eqref{item:basic_1_2}.
\end{proof}

\begin{defi}[$\overline{n}$ and $\underline{n}$]
Let $n=(n_1,\dots,n_r)$ be a sequence of integers.
Define $\overline{n}=(\overline{n}_1,\dots,\overline{n}_r)$ and $\underline{n}=(\underline{n}_1,\dots,\underline{n}_r)$ as
\[
\overline{n}_i=\min[n_i,n_{i+1},\dots,n_r],\quad
\underline{n}_i=\min[n_1,n_{2},\dots,n_i].
\]
If every element of $n$ is positive, $\overline{n}$ corresponds to the smallest Young diagram that contains $n$, and $\underline{n}$ corresponds to the largest Young diagram that is contained in $n$.
\end{defi}

\begin{cor}\label{cor:X(n)}\label{cor:x(n)}
Let $n=(n_1,\dots,n_r)$ be a sequence of integers with 
$n_i-n_{i+1}\geq -1$ for all $i$.
Then we have
\[
\ket{n}^G=(-\beta)^{\zet{\overline{n}}-\zet{n}}\ket{\overline{n}}^G,\qquad \ket{n}^g=(-\beta)^{\zet{n}-\zet{\underline{n}}}\ket{\underline{n}}^g,
\]
where $\zet{n}=\sum_{i=1}^r n_i$.
\end{cor}
\begin{proof}
This corollary follows directly from Lemma \ref{lemma:basic_2} (\ref{item:basic_2_1}--\ref{item:basic_2_3}).
\end{proof}

A \textit{rook strip} is a skew Young diagram that has at most one box in any row or column.
\begin{lemma}\label{lemma:rook}
Let $\lambda=(\lambda_1,\lambda_2,\dots,\lambda_r)$ be a partition.
Then we have
\begin{equation}
e^{\theta}\ket{\lambda}^g
=\sum_{\lambda/\mu\,\mathrm{rook\, strip}}
\beta^{\zet{\lambda/\mu}}\cdot 
\ket{\mu}^g.
\end{equation} 
This equation is equivalent to
\begin{equation}
{}^g\bra{\lambda}e^{\Theta}
=\sum_{\lambda/\mu\,\mathrm{rook\, strip}}
\beta^{\zet{\lambda/\mu}}\cdot 
{}^g\bra{\mu}.
\end{equation} 
\end{lemma}
\begin{proof}
From Lemma \ref{lemma:basic_1} (\ref{item:basic_1_2}) and $e^\theta \ket{-r}=\ket{-r}$, we have
\begin{align}
&e^{\theta}\ket{\lambda}^g\nonumber \\
&=
(\psi_{\lambda_1-1}+\beta \psi_{\lambda_1-2})e^\theta 
(\psi_{\lambda_2-2}+\beta \psi_{\lambda_2-3})e^\theta 
\dots
(\psi_{\lambda_r-r}+\beta \psi_{\lambda_r-r-1})e^\theta 
\ket{-r}^g\nonumber\\
&=\sum_{\epsilon_i\in \{0,1\}}\beta^{\sharp\{i\,\vert\,\epsilon_i=1\} }\cdot 
\ket{\lambda_1-\epsilon_1,\dots,\lambda_r-\epsilon_r}^g.
\label{eq:e^theta_ket_lambda}
\end{align}
If the sequence $\lambda_1\geq \lambda_2\geq \dots \geq \lambda_r$ contains some $p$ such that $\lambda_p=\lambda_{p+1}$, the contributions of
\[
\ket{\lambda_1-\epsilon_1,\dots,\lambda_p-1,\lambda_{p+1},\dots,\lambda_r-\epsilon_r }^g
\]
and 
\[
\ket{\lambda_1-\epsilon_1,\dots,\lambda_p-1,\lambda_{p+1}-1,\dots,\lambda_r-\epsilon_r }^g
\] 
in \eqref{eq:e^theta_ket_lambda} are canceled out because of Lemma \ref{lemma:basic_2} (\ref{item:basic_2_3}).
Therefore, the only terms $\ket{\mu_1,\dots,\mu_r}^g$ that survive must satisfy
\[
\lambda_i-\mu_i\in \{0,1\}\quad 
\mbox{and}\quad 
\alpha=\lambda_p=\lambda_{p+1}
\ \Rightarrow\ 
\alpha=\mu_p.
\]
This is equivalent to saying that $\lambda/\mu$ is a rook strip.
\end{proof}

\section{Stable Grothendieck polynomial and its dual}\label{sec:Gr}

\subsection{The completed ring $\widehat{\Lambda}$}\label{sec:completion}

Hereafter, we always assume $k=\CC(\beta)$.
Let $I_n$ ($n\geq 0$) be the $k$-subspace of $\Lambda$ that is defined as 
\[
I_n=\sum_{\ell(\lambda)>n}k\cdot s_{\lambda} (x),
\]
where $\ell(\lambda)$ is the length of a partition $\lambda$.
As $I_n\supset I_{n+1}$, there exists an inverse system
$
\Lambda/I_0\leftarrow \Lambda/I_1\leftarrow\Lambda/I_2 \leftarrow \cdots
$
of $k$-spaces.
Let $\widehat{\Lambda}:=\lim\limits_{\longleftarrow} (\Lambda/I_n)$ be the inverse limit.
There exists a natural inclusion $\Lambda\hookrightarrow \widehat{\Lambda}$.

We define the $k$-linear topology on $\Lambda$ where the family $\{I_n\}_{n\geq 0}$ forms an open neighborhood base at $0$.
The inclusion $\Lambda \hookrightarrow\widehat{\Lambda}$ can be viewed as a completion of the topological space $\Lambda$.
Note that
\begin{align}
f(x)\in I_{n} 
&\iff f(x_1,\dots,x_{n},0,0,\dots)=0\label{eq:M_n(1)}\\
&\iff \langle f(x),s_\lambda(x)\rangle =0\mbox{ for all } \ell(\lambda)\leq n.\label{eq:M_n(2)}
\end{align}

\begin{example}
$e_0+e_1+e_2+\cdots$ is contained in $\widehat{\Lambda}$, while $h_0+h_1+h_2+\cdots$ is not.
\end{example}

Moreover, $\widehat{\Lambda}$ is indeed a topological $k$-algebra, which has a continuous  multiplication as well as a continuous addition.

For a partition $\lambda$, we have
$
\langle I_n,s_\lambda\rangle =0
$
for any $\ell(\lambda)\leq n$.
This implies that the map $\Lambda\to k; f\mapsto \langle f,s_\lambda\rangle$ is continuous. Therefore, the bilinear form is uniquely extend to the continuous map:
\[
\langle \cdot,\cdot \rangle:\widehat{\Lambda}\times \Lambda \to k.
\]

Let $\Lambda_n:=k[x_1,\dots,x_n]^{\mathfrak{S}_n}$ be the $k$-algebra of symmetric polynomials in $n$ variables.
We often identify $\Lambda_n$ with $\Lambda/I_n$ by sending $s_\lambda(x_1,\dots,x_n)$ to $(s_\lambda(x)$ $\mathrm{mod}\,I_n)$ for $\ell(\lambda)\leq n$.
Let $\pi_n:\widehat{\Lambda}\to \Lambda/I_n\simeq \Lambda_n$ denote the natural projection, which coincides with the substitution map:
\[
\pi_n(f(x))=f(x_1,x_2,\dots,x_n,0,0,\dots).
\]

It is convenient to define the inclusion $\iota_n:\Lambda_n\hookrightarrow \widehat{\Lambda}$ that sends $s_\lambda(x_1,\dots,x_n)$ to $s_\lambda(x)$ for $\ell(\lambda)\leq n$.
Obviously, we have $\pi_n\circ \iota_n=\mathrm{id}_{\Lambda_n}$ and $\mathrm{Im}(\iota_n)\subset \Lambda\subsetneq \widehat{\Lambda} $.
For example, we have $\iota_n\circ \pi_n(e_0+e_1+e_2+\cdots)=e_0+e_1+\dots+e_n$.
We call the map 
\[
\iota_n\circ \pi_n:\widehat{\Lambda}\to \widehat{\Lambda}
\]
the \textit{$n$-th truncation map}.
It is also important to note $\mathrm{Im}(1-\iota_n\circ \pi_n)\subset I_n$.

\subsection{Free-fermionic presentation of stable Grothendieck polynomials}

A \textit{stable Grothendieck polynomial}~\cite{fomin1994grothendieck} is a unique element $G_\lambda(x)$ of $\widehat{\Lambda}$ that satisfies
\[
\pi_n(G_\lambda(x))=G_\lambda(x_1,\dots,x_n),\qquad\mbox{for all $n$}.
\]
In \cite{iwao2020freefermion}, the author proved the following presentation:
\begin{prop}[{\cite[Theorem 3.6]{iwao2020freefermion}}]\label{prop:main_Iwao_G}
We have
$
G_\lambda(x)=\bra{0}e^{H(x)}\ket{\lambda}^G
$.
\end{prop}
\begin{cor}[{\cite[Remark 2.8]{HUDSON2017115}}]\label{cor:generating_fn_G}
We have
\[
\sum_{n\in \ZZ}G_n(x)z^n=\frac{1}{1+\beta z^{-1}}\prod_{l=1}^\infty\frac{1+\beta x_l}{1-x_lz},
\]
where $G_n(x)=\bra{0}e^{H(x)}\ket{(n)}^G$.
The $(n)$ means a sequence of length one.
\end{cor}
\begin{cor}\label{cor:Grothen}
Let $\mathcal{G}(z)=\sum_{n\in \ZZ}G_n(x)z^n$.
Then we have
\[
e^{H(x)}e^{-\Theta}\psi(z)e^{\Theta}e^{-H(x)}
=\prod_{l=1}^\infty(1+\beta x_l)^{-1}\cdot \mathcal{G}(z)\psi(z).
\]
\end{cor}


For $\ell(\lambda)\leq r$, we consider the symmetric function $G_\lambda^r(x)\in \Lambda$ defined as
\[
G_\lambda^r(x):=\iota_r(G_\lambda(x_1,\dots,x_r))=\iota_r\circ \pi_r(G_\lambda(x)).
\]
We call $G_\lambda^r(x)$ the \textit{$r$-th truncation of $G_\lambda(x)$}.
\begin{prop}[{\cite[Proposition 3.1]{iwao2020freefermion}}]\label{main_Iwao_G_finite}
The $r$-th truncation of $G^r_\lambda(x)$ satisfies
\[
G^r_\lambda(x)=\bra{0}e^{H(x)}\psi_{\lambda_1-1}e^{\Theta}\psi_{\lambda_2-2}e^{\Theta}
\dots \psi_{\lambda_r-r}e^{\Theta}\cdot e^{-r\Theta}\ket{-r}.
\]
\end{prop}

\subsection{Free-fermionic presentation of dual stable Grothendieck polynomials}

The \textit{dual stable Grothendieck polynomial} $\{g_\lambda\}_\lambda\subset \Lambda$ is a family of symmetric functions that satisfies
\begin{equation}\label{eq:duality}
\langle G_\lambda,g_\mu\rangle =\delta_{\lambda,\mu}.
\end{equation}
Since $\Lambda$ is a dense subspace of $\widehat{\Lambda}$, the family $\{g_\lambda\}_\lambda$ is uniquely determined by \eqref{eq:duality}.

\begin{prop}[{\cite[Proposition 4.1]{iwao2020freefermion}}]\label{prop:main_Iwao_g}
We have
$
g_\lambda(x)=\bra{0}e^{H(x)}\ket{\lambda}^g
$.
\end{prop}

\subsection{Evaluation of remainder terms}

The following lemmas will be used in Sections \ref{sec:G-expan}--\ref{sec:expansion_of_s^perp} to evaluate ``remainder terms'' which are expected to be ``sufficiently small,'' or ``sufficiently near to $0$,'' in the topological $k$-linear space $\widehat{\Lambda}$.
\begin{lemma}\label{lemma:evaluate1}
Let $n_1\geq n_2\geq \dots\geq n_r>0$, $s>0$, and $i>0$.
Then the following $(\ref{item:evaluate1_1}$--$\ref{item:evaluate1_3})$ hold.
\begin{enumerate}
\item \label{item:evaluate1_1}
$\bra{0}e^{H(x)}\psi_{n_1-1}\psi_{n_2-2}\dots \psi_{n_r-r}\ket{-r}\in I_{r-1}$.
\item \label{item:evaluate1_2}
$\bra{0}e^{H(x)}\psi_{n_1-1}\psi_{n_2-2}\dots \psi_{n_r-r}e^{s\Theta}\ket{-r}\in I_{r-1}$.
\item \label{item:evaluate1_3}
$\bra{0}e^{H(x)}\psi_{n_1-1}e^{\Theta}\psi_{n_2-2}e^{\Theta}\dots \psi_{n_r-r}e^\Theta a_{-i}\ket{-r}\in I_r$.
\end{enumerate}
\end{lemma}
\begin{proof}
\eqref{item:evaluate1_1} follows from Theorem \ref{thm:Schur}.

\eqref{item:evaluate1_2}: Let
$
X_{-r}:=
\binom{s}{1}\beta\psi_{-r}+\binom{s}{2}\beta^2\psi_{-r+1}+\cdots
+\binom{s}{s}\beta^s\psi_{-r-1+s}
$.
Since $e^{s\Theta}\psi_{-r-1}=(\psi_{-r-1}+X_{-r})e^{s\Theta}$, we have
\begin{align*}
&\bra{0}e^{H(x)}\psi_{n_1-1}\dots \psi_{n_r-r}e^{s\Theta}\ket{-r}\\
&=
\bra{0}e^{H(x)}\psi_{n_1-1}\dots \psi_{n_r-r}\ket{-r}
+\bra{0}e^{H(x)}\psi_{n_1-1}\dots \psi_{n_r-r}X_{-r}\ket{-r-1}\nonumber\\
&+
\bra{0}e^{H(x)}\psi_{n_1-1}\dots \psi_{n_r-r}(\psi_{-r-1}+X_{-r})X_{-r-1}\ket{-r-2}\nonumber\\
&+
\bra{0}e^{H(x)}\psi_{n_1-1}\dots \psi_{n_r-r}(\psi_{-r-1}+X_{-r})(\psi_{-r-2}+X_{-r-1})X_{-r-2}\ket{-r-3}+\cdots
\end{align*}
By \eqref{item:evaluate1_1}, each term on the right hand side is contained $I_{r-1}$.
This concludes \eqref{item:evaluate1_2}.

\eqref{item:evaluate1_3}:
Since 
\begin{equation}\label{eq:a_-ir}
a_{-i}\ket{-r}
=
\sum_{p=1}^i(\psi_{-r-1}\psi_{-r-2}\dots \psi_{-r-p-1})\psi_{-r-p+i}\ket{-r-p},
\end{equation}
\eqref{item:evaluate1_3} follows from \eqref{item:evaluate1_2} and Lemma \ref{lemma:basic_1} \eqref{item:basic_1_1}.
\end{proof}

\begin{lemma}\label{lemma:evaluate2}
Let $r>0$ and $s\geq i>0$.
Then the following $(\ref{item:evaluate2_1}$--$\ref{item:evaluate2_2})$ hold
\begin{enumerate}
\item \label{item:evaluate2_1}
$\psi_{-r-1}\psi_{-r-2}\dots \psi_{-r-s} a_{-i}\ket{-r}=0$.
\item \label{item:evaluate2_2}
$\psi_{-r-1}e^{-\theta}\psi_{-r-2}\dots e^{-\theta}\psi_{-r-s} a_{-i}\ket{-r-s}=0$.
\end{enumerate}
\begin{proof}
\eqref{item:evaluate2_1} follows from \eqref{eq:a_-ir}.
\eqref{item:evaluate2_2} follows from the equation 
\begin{align*}
&\psi_{-r-1}e^{-\theta}\psi_{-r-2}\dots e^{-\theta}\psi_{-r-s} a_{-i}\ket{-r-s}\\
&=e^{-(s-1)\theta}\psi_{-r-1}\psi_{-r-2}\dots \psi_{-r-s} a_{-i}\ket{-r-s}.
\end{align*}
\end{proof}
\begin{rem}\label{rem:note}
To call the reader's attention, we would note
\[
\psi_{-r-1}e^{-\theta}\psi_{-r-2}\dots e^{-\theta}\psi_{-r-s}\mbox{\fbox{$e^{-\theta}$}} a_{-i}\ket{-r-s}\neq 0,
\]
in which the extra $e^{-\theta}$ at the boxed position is added to Lemma \ref{lemma:evaluate2} \eqref{item:evaluate2_2}.
\end{rem}

\end{lemma}

\section{Free-fermionic presentation of skew stable Grothendieck polynomials}\label{sec:skew}

In this section, we give a free-fermionic presentation of the stable Grothendieck polynomials corresponding to a skew diagram.

Let $\Delta:\widehat{\Lambda}\to \widehat{\Lambda}\otimes \widehat{\Lambda}$ be the coproduct on $\widehat{\Lambda}$ with $(\Delta f)(x,y)=f(x,y)$.
Buch \cite{buch2002littlewood} defined the symmetric function $G_{\lambda\slsl\mu}(x)$ such that
\begin{equation}\label{eq:def_of_G//}
\Delta (G_\lambda)=\sum_{\mu\subset \lambda}G_\mu\otimes G_{\lambda\slsl \mu},
\end{equation}
which is equivalent to $G_{\lambda\slsl\mu}(x)=g_\mu^\perp G_\lambda(x)$.
\begin{thm}[Free-fermionic presentation of $G_{\lambda\slsl \mu}$]\label{thm:G//}
We have
\[
G_{\lambda\slsl\mu}(x)={}^g\bra{\mu}e^{H(x)}\ket{\lambda}^G,\qquad \mbox{where}\quad
{}^g\bra{\mu}:=\omega(\ket{\mu}^g).
\]
\end{thm}
\begin{proof}
It immediately follows from Theorem \ref{thm:adjoint_action_fermion} and Propositions \ref{prop:main_Iwao_G}, \ref{prop:main_Iwao_g}.
\end{proof}

In \cite[\S 6]{buch2002littlewood}, it is proved the fact that the \textit{skew stable Grothendieck polynomials} $G_{\lambda/\mu}(x)$ are unique symmetric functions that satisfy
\[
G_{\lambda\slsl\mu}(x)=\sum_{\mu/\sigma\,\mathrm{rook\, strip}}\beta^{\zet{\mu/\sigma}}G_{\lambda/\sigma}(x).
\]

\begin{thm}[Free-fermionic presentation of $G_{\lambda/ \mu}$]\label{thm:G/}
We have
\[
G_{\lambda/\mu}(x)={}^g\bra{\mu}e^{-\Theta}e^{H(x)}\ket{\lambda}^G.
\]
\end{thm}
\begin{proof}
It is a direct consequence of Lemma \ref{lemma:rook}.
\end{proof}

\subsection{Determinantal formula for $G_{\lambda\slsl \mu}(x)$}\label{sec:det_G//}

Once a free-fermionic presentation have been obtained, it is a straightforward task to derive a determinantal formula.

Let
\begin{align}
&\Phi=\Phi(z_1,\dots,z_r,w_1,\dots,w_r)\nonumber\\
&:=
\bra{-r}e^{-\Theta}\psi^\ast(w_r)\dots e^{-\Theta}\psi^\ast(w_1)
e^{H(x)}
\psi(z_1)e^\Theta  \dots \psi(z_r)e^\Theta\ket{-r}
\label{eq:Phi_expectation_value}
\end{align}
be a ``generating function'' of the skew stable Grothendieck polynomials, in which the coefficient of $z_1^{\lambda_1-1}\dots z_r^{\lambda_r-r}w_1^{\mu_1-1}\dots w_r^{\mu_r-r}$ is $G_{\lambda\slsl\mu}(x)$.

By letting
\[
\begin{array}{lll}
A_i&:=e^{-(r-i+1)\Theta }\psi(z_i)e^{(r-i+1)\Theta}&=(1+\beta z_i^{-1})^{-(r-i+1)}\psi(z_i),\\
B_i&:=e^{-(r-i+1)\Theta }\psi^\ast(w_i)e^{(r-i+1)\Theta}&=(1+\beta w_i)^{r-i+1}\psi^\ast(w_i),\\
X&:=e^{-r\Theta}e^{H(x)}e^{r\Theta}&=\textstyle \prod_{l=1}^\infty(1+\beta x_l)^re^{H(x)},
\end{array}
\]
the $\Phi$ is simply rewritten as
\[
\bra{-r}B_r\dots B_2B_1XA_1A_2\dots A_r\ket{-r}.
\]
Since $e^{H(x)}\ket{-r}=\ket{-r}$, this expression equals to
\begin{align}
\textstyle \prod_{l=1}^\infty(1+\beta x_l)^r\bra{-r}B_r\dots B_2B_1
(e^{H(x)}A_1e^{-H(x)})
\dots 
(e^{H(x)}A_re^{-H(x)})\ket{-r}.\label{eq:Phi_expand}
\end{align}
Further, by Wick's theorem (Theorem \ref{thm:Wick}), \eqref{eq:Phi_expand} is rewritten as
\begin{align}\label{eq:Phi_expand2}
\textstyle \prod_{i=1}^\infty(1+\beta x_i)^r
\det (\bra{-r}B_je^{H(x)}A_ie^{-H(x)}\ket{-r})_{1\leq i,j\leq r}.
\end{align}
By using the relation
\begin{align*}
e^{H(x)}A_ie^{-H(x)}
&=e^{H(x)}e^{-(r-i+1)\Theta }\psi(z_i)e^{(r-i+1)\Theta}e^{-H(x)}\\
&\stackrel{\mbox{\hspace{-.75em}\hbox to 0pt{\scriptsize \eqref{eq:e_theta_action} }}}{=}
\hspace{0em}
(1+\beta z_i)^{-(r-i)}e^{H(x)}e^{-\Theta }\psi(z_i)e^{\Theta}e^{-H(x)}\\
&
\stackrel{\mbox{\hspace{-3em}\hbox to 0pt{\scriptsize (Cor.\ref{cor:Grothen}) }}}{=}
\hspace{1em}
\textstyle \prod_{l=1}^\infty(1+\beta x_l)^{-1}\cdot(1+\beta z_i^{-1})^{-(r-i)} \mathcal{G}(z_i)\psi(z_i),
\end{align*}
we can cancel out the factor $\prod_{l=1}^\infty(1+\beta x_l)^r$ in (\ref{eq:Phi_expand2}) and obtain 
\begin{align}
\Phi
&=
\det \left( (1+\beta w_j)^{r-j+1}(1+\beta z_i^{-1})^{-(r-i)}\mathcal{G}(z_i)\cdot \bra{-r}\psi^\ast(w_j)\psi(z_i)\ket{-r} \right)_{1\leq i,j\leq r}\nonumber\\
&=
\det \left( (1+\beta w_j)^{r-j+1}(1+\beta z_i^{-1})^{-(r-i)}\mathcal{G}(z_i)
\sum_{p=-r}^{\infty}{z_i^pw_j^p} \right)_{1\leq i,j\leq r}.\label{eq:Psi-determinant}
\end{align}
\begin{prop}\label{prop:det_G//}
We have
\[
G_{\lambda\slsl \mu}(x)=
\det 
\left( 
\sum_{n=0}^\infty
\left\{
{i-j+1\choose n}
-
\Delta_{i,j}(n)
\right\}
\beta^n
 G_{\lambda_i-\mu_j-i+j+n}(x)
\right)_{1\leq i,j\leq r},
\]
where
\[
\Delta_{i,j}(n)
=
\begin{cases}
{i-r\choose -1-i+j+n}, & \mu_j=0,\\
0, & \mu_j> 0.
\end{cases}
\]
\end{prop}
\begin{proof}
By comparing the coefficients of $z_1^{\lambda_1-1}\dots z_r^{\lambda_r-r}w_1^{\mu_1-1}\dots w_r^{\mu_r-r}$ in (\ref{eq:Psi-determinant}), we obtain
\[
G_{\lambda\slsl \mu}(x)=
\det \left( \sum_{n=0}^\infty
\sum_{k=0}^{\mu_j-j+r}
{r-j+1\choose k}
{i-r\choose n-k}
\beta^n
 G_{\lambda_i-\mu_j-i+j+n}(x)
\right)_{1\leq i,j\leq r}.
\]
Therefore, to conclude the proof, it suffices to prove
\[
\sum_{k=0}^{\mu_j-j+r}
{r-j+1\choose k}
{i-r\choose n-k}
=
{i-j+1\choose n}
-
\Delta_{i,j}(n),
\]
which is given by direct calculations.
\end{proof}

\subsection{Determinantal formula for $G_{\lambda/\mu}(x)$}

The same calculations as in the previous section are also valid for $G_{\lambda/\mu}(x)$.
Let 
\begin{align*}
&\overline{\Phi}=\overline{\Phi}(z_1,\dots,z_r,w_1,\dots,w_r)\\
&:=
\bra{-r}e^{-\Theta}\psi^\ast(w_r)\dots e^{-\Theta}\psi^\ast(w_1)
e^{-\Theta}e^{H(x)}
\psi(z_1)e^\Theta  \dots \psi(z_r)e^\Theta\ket{-r},
\end{align*}
be a generating function of the skew stable Grothendieck polynomials, in which the coefficient of $z_1^{\lambda_1-1}\dots z_r^{\lambda_r-r}w_1^{\mu_1-1}\dots w_r^{\mu_r-r}$ is $G_{\lambda/\mu}(x)$.
By putting $$\overline{B}_i:=e^{-(r-i)\Theta }\psi^\ast(w_i)e^{(r-i)\Theta}=(1+\beta w_i)^{r-i}\psi^\ast(w_i),$$ we have
\begin{align}
&\overline{\Phi}=\bra{-r}\overline{B}_r\dots \overline{B}_2\overline{B}_1XA_1A_2\dots A_r\ket{-r}.
\end{align}
Using the same calculations as in \S \ref{sec:det_G//}, we obtain the determinantal expression
\begin{align*}
\overline{\Phi}
=
\det \left( (1+\beta w_j)^{r-j}(1+\beta z_i^{-1})^{-(r-i)}\mathcal{G}(z_i)
\sum_{p=-r}^{\infty}{z_i^pw_j^p} \right)_{1\leq i,j\leq r}.
\end{align*}
Comparing the coefficients of $z_1^{\lambda_1-1}\dots z_r^{\lambda_r-r}w_1^{\mu_1-1}\dots w_r^{\mu_r-r}$, we obtain:
\begin{prop}\label{prop:det_G/}
We have
\[
G_{\lambda/ \mu}(x)=
\det \left( \sum_{n=0}^\infty
{i-j\choose n}
\beta^n
 G_{\lambda_i-\mu_j-i+j+n}(x)
\right)_{1\leq i,j\leq r}.
\]
\end{prop}
\begin{proof}
It is given by direct calculations using the equation
\[
\sum_{k=0}^{\mu_j-j+r}
{r-j\choose k}
{i-r\choose n-k}
=
{i-j\choose n}.
\]
\end{proof}

\subsection{Determinantal formula for $G_{\lambda\slsl \mu}(x_1,\dots,x_m)$}\label{sec:det_G//_finite}

We have another determinantal formula which describes skew Grothendieck polynomials in finite variables.
We put
\begin{align*}
&G_{\lambda\slsl\mu}(x_1,\dots,x_m)=G_{\lambda\slsl\mu}(x_1,\dots,x_m,0,0,\dots),\\
&G_{\lambda/\mu}(x_1,\dots,x_m)=G_{\lambda/\mu}(x_1,\dots,x_m,0,0,\dots).
\end{align*}

\begin{lemma}\label{lemma:adjoint_distance}
Let $f(x)\in \widehat{\Lambda}$.
Then $f(x)\in I_n$ if and only if $s_\lambda^\perp f(x)\in I_{n-\ell(\lambda)}$.
\end{lemma}
\begin{proof}
From (\ref{eq:M_n(2)}), it follows that
$f(x)\in I_n$
$\iff$
$\langle f,s_\mu\rangle=0$ for all $\ell(\mu)\leq n$
$\iff$
$\langle f,s_\lambda s_\mu\rangle=0$ for all $\ell(\lambda)+\ell(\mu)\leq n$
$\iff$
$\langle s_\lambda^\perp f,s_\mu\rangle=0$ for all $\ell(\mu)\leq n-\ell(\lambda)$
$\iff$
$s_\lambda^\perp f(x)\in I_{n-\ell(\lambda)}$.
\end{proof}
\begin{cor}\label{cor:g_adjoint_distance}
$f(x)\in I_n$ if and only if $g_\lambda^\perp f(x)\in I_{n-\ell(\lambda)}$.
\end{cor}

\begin{prop}
Let $G_\lambda^r(x)$ be the $r$-th truncation of $G_\lambda(x)$. 
For a partition $\mu$ and an integer $m$ with $m\leq r-\ell(\mu)$, we have
\begin{equation}\label{eq:skew_truncation}
G_{\lambda\slsl\mu}(x_1,\dots,x_m)=g_\mu^\perp G_\lambda^r(x_1,\dots,x_m),
\end{equation}
where $g_\mu^\perp G_\lambda^r(x_1,\dots,x_m)=g_\mu^\perp G_\lambda^r(x_1,\dots,x_m,0,0,\dots)$.
\end{prop}
\begin{proof}
As $\mathrm{Im}(1-\iota_r\circ \pi_r)\subset I_r$, we have $G_\lambda(x)-G_\lambda^r(x)\in I_r$. 
By Corollary \ref{cor:g_adjoint_distance}, it follows that $g_\mu^\perp G_\lambda(x)-g_\mu^\perp G_\lambda^r(x)\in I_{r-\ell(\mu)}$.
Therefore, $G_{\lambda\slsl\mu}(x_1,\dots,x_{r-\ell(\mu)})=g_\mu^\perp G_\lambda^r(x_1,\dots,x_{r-\ell(\mu)})$.
See (\ref{eq:M_n(1)}).
\end{proof}

Consider the formal series
\begin{align*}
&\Psi=\Psi(z_1,\dots,z_r,w_1,\dots,w_r)\nonumber\\
&:=
\bra{-r}e^{-\Theta}\psi^\ast(w_r)\dots e^{-\Theta}\psi^\ast(w_1)
e^{H(x)}
\psi(z_1)e^\Theta  \dots \psi(z_r)e^\Theta\cdot e^{-r\Theta}\ket{-r},
\end{align*}
in which the coefficient of $z_1^{\lambda_1-1}\dots z_r^{\lambda_r-r}w_1^{\mu_1-1}\dots w_r^{\mu_r-r}$ is $g_\mu^\perp G_{\lambda}^r(x)$.
By letting
\[
\begin{array}{lll}
C_i&:=e^{(i-1)\Theta }\psi(z_i)e^{-(i-1)\Theta}&=(1+\beta z_i^{-1})^{i-1}\psi(z_i),\\
D_i&:=e^{(i-1)\Theta }\psi^\ast(w_i)e^{-(i-1)\Theta}&=(1+\beta w_i)^{-(i-1)}\psi^\ast(w_i),
\end{array}
\]
we obtain
\begin{align*}
\Psi
&=\bra{-r}D_r\dots D_2D_1e^{H(x)}C_1C_2\dots C_r\ket{-r}\nonumber\\
&=\bra{-r}D_r\dots D_2D_1
(e^{H(x)}C_1e^{-H(x)})
\dots 
(e^{H(x)}C_re^{-H(x)})\ket{-r}\\
&\hspace{-1.3em}
\stackrel{
\mbox{\scriptsize (Thm.\ref{thm:Wick})}
}{=}
\det (\bra{-r}D_je^{H(x)}C_ie^{-H(x)}\ket{-r})_{1\leq i,j\leq r}.
\end{align*}
From
\begin{align*}
e^{H(x)}C_ie^{-H(x)}
=(1+\beta z_i^{-1})^{i-1}e^{H(x)}\psi(z_i)e^{-H(x)}
=(1+\beta z_i^{-1})^{i-1}\mathcal{H}(z_i)\psi(z_i),
\end{align*}
it follows that
\begin{align*}
\Psi&=
\det \left( (1+\beta w_j)^{-(j-1)}(1+\beta z_i^{-1})^{i-1}\mathcal{H}(z_i)\cdot \bra{-r}\psi^\ast(w_j)\psi(z_i)\ket{-r} \right)_{1\leq i,j\leq r}\\
&=
\det \left( (1+\beta w_j)^{-(j-1)}(1+\beta z_i^{-1})^{i-1}\mathcal{H}(z_i)
\sum_{p=-r}^{\infty}{z_i^pw_j^p} \right)_{1\leq i,j\leq r}.
\end{align*}
Comparing the coefficient of $z_1^{\lambda_1-1}\dots z_r^{\lambda_r-r}w_1^{\mu_1-1}\dots w_r^{\mu_r-r}$, we have:
\begin{prop}\label{prop:det_G//_finite}
We have
\[
g_\mu^\perp G^r_{\lambda}(x)=
\det \left( \sum_{n=0}^\infty
\sum_{k=0}^{\mu_j-j+r}
{1-j\choose k}
{i-1\choose n-k}
\beta^n
 h_{\lambda_i-\mu_j-i+j+n}(x)
\right)_{1\leq i,j\leq r}.
\]
\end{prop}

By letting
\[
H_p^{(i)}(x_1,\dots,x_m):=\sum_{n\geq 0}{i\choose n}
\beta^n h_{p+n}(x_1,\dots,x_m),\quad i\geq 0,
\]
Proposition \ref{prop:det_G//_finite} is simply rewritten as
\begin{align*}
G_{\lambda\slsl \mu}(x_1,\dots,x_m)=\det \left( \sum_{k=0}^{\mu_j-j+r}
{1-j\choose k}
H^{(i-1)}_{\lambda_i-\mu_j-i+j+k}(x_1,\dots,x_m)
\right)_{1\leq i,j\leq r}
\end{align*}
for $m\leq r-\ell(\mu)$.

\subsection{Determinantal formula for $G_{\lambda/\mu}(x_1,\dots,x_m)$}

Consider the formal series
\begin{align}
&\overline{\Psi}=\overline{\Psi}(z_1,\dots,z_r,w_1,\dots,w_r)\nonumber\\
&:=
\bra{-r}e^{-\Theta}\psi^\ast(w_r)\dots e^{-\Theta}\psi^\ast(w_1)
e^{-\Theta}
e^{H(x)}
\psi(z_1)e^\Theta  \dots \psi(z_r)e^\Theta\cdot e^{-r\Theta}\ket{-r},\label{eq:def_of_Psi}
\end{align}
in which the coefficient of $z_1^{\lambda_1-1}\dots z_r^{\lambda_r-r}w_1^{\mu_1-1}\dots w_r^{\mu_r-r}$ is $\sum_{\mu/\sigma \mathrm{\,rook\,strip}} g_\sigma^\perp G_{\lambda}^r(x)$.
By letting $\overline{D}_i:=e^{i\Theta }\psi^\ast(w_i)e^{-i\Theta}=(1+\beta w_i)^{-i}\psi^\ast(w_i)$, \eqref{eq:def_of_Psi} is written as
\begin{align*}
\overline{\Psi}=\bra{-r}\overline{D}_r\dots \overline{D}_2\overline{D}_1
(e^{H(x)}C_1e^{-H(x)})
\dots 
(e^{H(x)}C_re^{-H(x)})\ket{-r}.
\end{align*}
Similar calculations as before give
\begin{align*}
&\overline{\Psi}
=
\det \left( (1+\beta w_j)^{-j}(1+\beta z_i^{-1})^{i-1}\mathcal{H}(z_i)
\sum_{p=-r}^{\infty}{z_i^pw_j^p} \right)_{1\leq i,j\leq r},
\end{align*}
which implies:
\begin{prop}\label{prop:det_G/_finite}
For $m\leq r-\ell(\mu)$, we have
\[
G_{\lambda/\mu}(x_1,\dots,x_m)
=\det \left( \sum_{k=0}^{\mu_j-j+r}
{-j\choose k}
H^{(i-1)}_{\lambda_i-\mu_j-i+j+k}(x_1,\dots,x_m)
\right)_{1\leq i,j\leq r}.
\]
\end{prop}

\section{Free-fermionic presentation of skew dual stable Grothendieck polynomials}\label{sec:skewdual}

\subsection{Determinantal formula for $g_{\lambda/\mu}(x)$}

In this section, we study the \textit{skew dual stable Grothendieck polynomials} $g_{\lambda/\mu}(x)$.
It is known (see, for example, Yeliussizov~\cite{yeliussizov2017duality}) that 
\[
\Delta(g_\lambda)=\sum_{\mu\subset \lambda}g_\mu\otimes g_{\lambda/ \mu},
\]
which is equivalent to $g_{\lambda/\mu}(x)=G_\mu^\perp g_\lambda(x)$.
From this, we immediately obtain the following presentation:
\begin{thm}[A free-fermionic presentation of $g_{\lambda/\mu}$]\label{thm:g/}
We have
\[
g_{\lambda/\mu}(x)={}^G\bra{\mu}e^{H(x)}\ket{\lambda}^g.
\]
\end{thm}

Let us proceed to a determinantal formula.
Consider the formal series
\begin{align*}
&\Omega=\Omega(z_1,\dots,z_r,w_1,\dots,w_r)\\
&:=
\bra{-r}e^{\theta}\psi^\ast(w_r)\dots e^{\theta}\psi^\ast(w_1)
e^{H(x)}
\psi(z_1)e^{-\theta}  \dots \psi(z_r)e^{-\theta}\ket{-r},
\end{align*}
whose coefficient of $z_1^{\lambda_1-1}\dots z_r^{\lambda_r-r}w_1^{\mu_1-1}\dots w_r^{\mu_r-r}$ is $g_{\lambda/\mu}(x)$.
By letting
\[
\begin{array}{lll}
P_i&:=e^{(r-i+1)\theta }\psi(z_i)e^{-(r-i+1)\theta}&=(1+\beta z_i)^{r-i+1}\psi(z_i),\\
Q_i&:=e^{(r-i+1)\theta }\psi^\ast(w_i)e^{-(r-i+1)\theta}&=(1+\beta w_i^{-1})^{-(r-i+1)}\psi^\ast(w_i),
\end{array}
\]
we obtain
\begin{align}
\Omega=
\bra{-r}Q_r\dots Q_2Q_1
(e^{H(x)}P_1e^{-H(x)})
\dots 
(e^{H(x)}P_re^{-H(x)})\ket{-r}.\label{eq:Omega_expand}
\end{align}
Using Wick's theorem (Theorem \ref{thm:Wick}), we rewrite \eqref{eq:Omega_expand} as 
\begin{align*}
\det (\bra{-r}Q_je^{H(x)}P_ie^{-H(x)}\ket{-r})_{1\leq i,j\leq r}.
\end{align*}
Since
$e^{H(x)}P_ie^{-H(x)}
=
(1+\beta z_i)^{r-i+1}\mathcal{H}(z_i)\psi(z_i)$,
we have
\begin{align*}
\Omega
=
\det \left( (1+\beta w_j^{-1})^{-(r-j+1)}(1+\beta z_i)^{r-i+1}\mathcal{H}(z_i)
\sum_{p=-r}^{\infty}{z_i^pw_j^p} \right)_{1\leq i,j\leq r}.
\end{align*}
Comparing the coefficients of $z_1^{\lambda_1-1}\dots z_r^{\lambda_r-r}w_1^{\mu_1-1}\dots w_r^{\mu_r-r}$, we obtain:
\begin{prop}\label{prop:det_g}
For $r\geq \ell(\lambda)$, we have
\[
g_{\lambda/ \mu}(x)=
\det 
\left( 
\sum_{n=0}^\infty
{j-i\choose n}
\beta^n
h_{\lambda_i-\mu_j-i+j-n}(x)
\right)_{1\leq i,j\leq r}.
\]
\end{prop}

For $i\in \ZZ$, let
\[
h_p^{(i)}(x)=\sum_{k=0}^{\infty}{i\choose k}
\beta^k h_{p-k}(x).
\]
Then Proposition \ref{prop:det_g} is rewritten as 
\[
g_{\lambda/ \mu}(x)=
\det 
\left( 
h^{(j-i)}_{\lambda_i-\mu_j-i+j}(x)
\right)_{1\leq i,j\leq r}.
\]

\begin{rem}\label{rem:various}
By using the same technique as before, we can also derive determinantal formulas for the symmetric functions
\[
s_\mu^\perp G_\lambda(x)=\bra{\mu}e^{H(x)}\ket{\lambda}^G,\quad
s_\mu^\perp g_\lambda(x)=\bra{\mu}e^{H(x)}\ket{\lambda}^g,\quad
g_\mu^\perp g_\lambda(x)={}^g\bra{\mu}e^{H(x)}\ket{\lambda}^g.
\]
In fact, we have
\begin{align}
&s_\mu^\perp G_\lambda(x)
=
\det\left(
\sum_{k=0}^{\infty}\sum_{n=0}^{\mu_j-j+r}
{r\choose n}
{i-r\choose k-n}
\beta^kG_{\lambda_i-\mu_j-i+j-k}(x)
\right)_{1\leq i,j\leq r},\\
	&\begin{aligned}
	s_\mu^\perp G_\lambda(x_1,\dots,x_m)
	=\det\left(
	H^{(i-1)}_{\lambda_i-\mu_j-i+j}(x_1,\dots,x_m)
	\right)_{1\leq i,j\leq r}
	,
	\end{aligned}\label{eq:s^muG}
\\
	&\begin{aligned}
	s_\mu^\perp g_\lambda(x)
	&
	=\det\left(
	h^{(1-i)}_{\lambda_i-\mu_j-i+j}(x)
	\right)_{1\leq i,j\leq r},
	\end{aligned}\label{eq:s^mug}
	\\
	&\begin{aligned}
	g_\mu^\perp g_\lambda(x)
	=\det\left(
	\sum_{n=0}^{\mu_j-j+r}
	{1-j\choose n}
	h^{(1-i)}_{\lambda_i-\mu_j-i+j+2n}(x)
	\right)_{1\leq i,j\leq r}.
	\end{aligned}
\end{align}
We do not give proofs here but leave them to the reader because they are quite similar to those that we have given before.
Note that this list does not contain ``$G_\mu^\perp G_\lambda$" since the bilinear form ``$\langle G_\mu,G_\lambda\rangle$" is undefinable.
\end{rem}

\begin{rem}
By substituting $(x_1,\dots,x_m)=(0,\dots,0)$ and $(x_1,x_2,\dots)=(0,0,\dots)$ to $(\ref{eq:s^muG}$--$\ref{eq:s^mug})$, we have
\begin{gather*}
s_\mu^\perp G_\lambda(0,\dots,0)
=
\det\left(
\beta^{\mu_j-\lambda_i-j+i}
{i-1\choose \mu_j-\lambda_i-j+i}
\right)_{1\leq i,j\leq r},\\
s_\mu^\perp g_\lambda(0)
=
\det\left(
\beta^{\lambda_i-\mu_j-i+j}
{1-i\choose \lambda_i-\mu_j-i+j}
\right)_{1\leq i,j\leq r}.
\end{gather*}
Since $\langle  f,s_\mu \rangle=\langle s_\mu^\perp f,1 \rangle =s_\mu^\perp f(0)$ for $f\in \widehat{\Lambda}$, we obtain
\begin{align*}
&G_\lambda(x_1,\dots,x_m)=\sum_{\mu\supset \lambda}
\beta^{\zet {\mu/\lambda}}
\det\left(
{i-1\choose \mu_j-\lambda_i-j+i}
\right)_{1\leq i,j\leq r}
s_\mu(x_1,\dots,x_m),\\
&g_\lambda(x)=\sum_{\mu\subset \lambda}
\beta^{\zet{\lambda/\mu}}
\det\left(
{1-i\choose \lambda_i-\mu_j-i+j}
\right)_{1\leq i,j\leq r}
s_\mu(x).
\end{align*}
These formulas are known to have an interesting combinatorial interpretation.
See \cite{lam2007combinatorial,lenart2000combinatorial}.
\end{rem}

\section{Non-commutative skew Schur function}\label{sec:noncommSchur}

In the following sections, we study combinatorial aspects of the skew Grothendieck polynomials.
Our main tool is the non-commutative Schur function~\cite{fomin1998noncommutative}.



\subsection{$s_{\lambda/\mu}(\bm{u})$ and $s_{\lambda/\mu}(\bm{V})$}

Let 
$
\mathfrak{X}:=\bigoplus_{\lambda}\QQ[\beta]\cdot \lambda
$
be the $\QQ[\beta]$-module freely generated by all partitions $\lambda$.

\begin{defi}
Let $\ee_i=(0,\dots,\stackrel{\stackrel{i}{\vee}}{1},\dots,0)$ be the $i$-th fundamental vector.
We define the linear operators $u_i:\mathfrak{X}\to \mathfrak{X}$ and $V_i:\mathfrak{X}\to \mathfrak{X}$ that act on a basis $\lambda$ as
\begin{align*}
&u_i\cdot \lambda=(-\beta)^{\zet{\overline{\lambda+\ee_i}}-\zet{\lambda+\ee_i} }\cdot \overline{\lambda+\ee_i},\\
&V_i\cdot \lambda=
\begin{cases}
\lambda\setminus \{\mbox{a box in $i$-th row} \}, & \mbox{if possible},\\
-\beta \cdot \lambda, & \mbox{otherwise}.
\end{cases}
\end{align*}
\end{defi}
By seeing their actions on the basis, one proves that the operators $u_i$ satisfy the following \textit{Knuth relation}:
\begin{eqnarray*}
\begin{gathered}
u_iu_ku_j=u_ku_iu_j,\quad i\leq j<k,\\
u_ju_iu_k=u_ju_ku_i,\quad i<j\leq k.
\end{gathered} 
\end{eqnarray*}
On the other hand, the operators $V_i$ satisfy the ``reverse'' Knuth relation:
\begin{eqnarray*}
\begin{gathered}
V_iV_kV_j=V_kV_iV_j,\quad i\geq j>k,\\
V_jV_iV_k=V_jV_kV_i,\quad i>j\geq k.
\end{gathered}
\end{eqnarray*}

We consider two types of characters $1,2,3,\dots,1',2',3',\dots$ following the order $1<2<3<\cdots$ and $1'>2'>3'>\cdots$.
Let $T$ be a \textit{skew tableau} \cite{fulton_1996} where each box contains a symbol $i$.
The \textit{column word} of $T$ is the sequence $w_T$ obtained by reading the entries of $T$ from bottom to top in each column, starting in the left column and moving to the right.
For example, $w_T=4331211$ for $T=\ygtab{\bl&\bl&1&1\\\bl&1&2\\3&3\\4}$.
Let $u^T$ be the (non-commutative) monomial defined by
\[
u^T:=u_{w_T(1)}u_{w_T(2)}\dots u_{w_T(N)}.
\]
We have $u^T=u_4u_3u_3u_1u_2u_1u_1$ in the above example.

Similarly, we define the non-commutative monomial $V^{T'}$ with $T'$ a skew tableau where each box contains an entry $j'$.
For example, $V^{T'}=V_1V_2V_3V_3V_4V_3V_1$
for 
$T'=\ygtab{\bl&\bl& 4'&3'&1'\\\bl&3'&3'\\\bl &2'\\1'}$.

\begin{defi}[Non-commutative skew Schur function]\label{defi:noncommutativeSchur}
For a skew shape $\lambda/\mu$, we define the non-commutative skew Schur functions
$s_{\lambda/\mu}(\bm{u}_m)=s_{\lambda/\mu}(u_1,\dots,u_m)$ and $s_{\lambda/\mu}(\bm{V}_n)=s_{\lambda/\mu}(V_1,\dots,V_n)$
as
\[
s_{\lambda/\mu}(\bm{u}_m):=\sum_{\substack{
T\, \mathrm{is}\,\mathrm{of}\, \mathrm{shape}\, \lambda/\mu,\\ 
\mathrm{Each}\, \mathrm{entry}\, \mathrm{of}\, T\, \mathrm{is\, in}\,  \{1,2,\dots,m\}}
}
{u^T}
\]
and
\[
s_{\lambda/\mu}(\bm{V}_n):=\sum_{\substack{
T'\, \mathrm{is}\,\mathrm{of}\, \mathrm{shape}\, \lambda/\mu,\\ 
\mathrm{Each}\, \mathrm{entry}\, \mathrm{of}\, T'\, \mathrm{is\, in}\,  \{1',2',\dots,n'\}}
}
{V^{T'}}.
\]
If $\mu=\emptyset$, we write $s_{\lambda}(\bm{u}_m)=s_{\lambda/\emptyset}(\bm{u}_m)$ and $s_{\lambda}(\bm{V}_m)=s_{\lambda/\emptyset}(\bm{V}_m)$.
\end{defi}

The following theorem is given by Fomin-Greene~\cite[Theorem 1.1]{fomin1998noncommutative}.
\begin{thm}\label{prop:noncommutativeSchur}
Let $u_1,\dots,u_n$ be a set of non-commutative operators that satisfies the Knuth relation.
Then the canonical map $s_{\lambda/\mu}\mapsto s_{\lambda/\mu}(\bm{u}_n)$ extends to a homomorphism from $\Lambda_n$ to the algebra $\Lambda_n(\bm{u})$ generated by the non-commutative Schur functions $s_\lambda(\bm{u}_n)$.
In particular, the $s_{\lambda}(\bm{u}_n)$ commute and the $s_{\lambda/\mu}(\bm{u}_n)$ expand according to the usual Littlewood-Richardson rule.
\end{thm}

\subsection{Actions on $\ket{n}^G$}

Let $n=(n_1,\dots,n_r)\in \ZZ^r$ and $i>0$.
By using \eqref{eq:relation_added} and the commutative relations $[a_{-i},e^\Theta]=0$, $[a_{i},e^{\Theta}]=-(-\beta)^ie^{\Theta}$, we obtain
\begin{align}
&a_{-i}\ket{n}^G=\sum_{j=1}^r\ket{n+i\ee_j }^G+\ket{\delta},\qquad \ket{\delta}=\psi_{n_1-1}e^\Theta\dots  \psi_{n_r-r}e^\Theta a_{-i}\ket{-r}
\label{eq:a_-i_on_G},\\
&\{a_{i}+r(-\beta)^i\}\ket{n}^G=\sum_{j=1}^r\ket{n-i\ee_j }^G\label{eq:a_i_on_G}.
\end{align}
The vector $\ket{\delta}$ will be seen as a ``remainder term," that is, a sufficiently small element in the topological ring $\widehat{\Lambda}$.

Let $E_i(P_1,P_2,\dots)$ and $E^{(r)}(P_1,P_2,\dots)$ be the polynomials in $P_1,P_2,\dots$ that satisfy
\[
e_i(x)=E_i(p_1(x),p_2(x),\dots),\quad
e_i(\overbrace{-\beta,\dots,-\beta}^r,x)=E^{(r)}_i(p_1(x),p_2(x),\dots).
\]
From (\ref{eq:a_-i_on_G}--\ref{eq:a_i_on_G}), we obtain
\begin{align*}
&E_i(a_{-1},a_{-2},\dots)\ket{n}^G=\sum_{1\leq j_1< \dots< j_i\leq r}\ket{n+\ee_{j_1}+\dots+\ee_{j_i} }^G+\ket{\Delta},\\
&E^{(r)}_i(a_{1},a_{2},\dots)\ket{n}^G=\sum_{1\leq j_1< \dots< j_i\leq r}\ket{n-\ee_{j_1}-\dots-\ee_{j_i} }^G,
\end{align*}
where $\ket{\Delta}$ is a finite sum of the vectors
\begin{equation}\label{eq:remainder_term}
\psi_{m_1-1}e^\Theta\dots \psi_{m_r-r}e^\Theta a_{-i_1}a_{-i_2}\dots a_{-i_s}\ket{-r},
\end{equation}
with $m_p\geq n_p$ ($p=1,2,\dots,r$) and $i_1,i_2,\dots,i_s>0$ ($s>0$).

If $n=\lambda$ is a partition, the $\ket{\lambda+\ee_{j_1}+\dots+\ee_{j_i} }^G$ and $\ket{\lambda-\ee_{j_1}-\dots-\ee_{j_i} }^G$ satisfy the assumption of Corollary \ref{cor:X(n)}.
Therefore, they can be written as 
\begin{align*}
&\ket{\lambda+\ee_{j_1}+\dots+\ee_{j_i} }^G=(-\beta)^{\zet{\overline{\lambda+\ee_{j_1}+\dots+\ee_{j_i}}}-\zet{\lambda+\ee_{j_1}+\dots+\ee_{j_i}}}\ket{\overline{\lambda+\ee_{j_1}+\dots+\ee_{j_i}} }^G,\\
&\ket{\lambda-\ee_{j_1}-\dots-\ee_{j_i} }^G=(-\beta)^{\zet{\overline{\lambda-\ee_{j_1}-\dots-\ee_{j_i}}}-\zet{\lambda-\ee_{j_1}-\dots-\ee_{j_i}}}\ket{\overline{\lambda-\ee_{j_1}-\dots-\ee_{j_i}} }^G.
\end{align*}
Further, by seeing the actions of $u_i$ and $V_i$, they are rewritten as
\begin{align*}
&\ket{\lambda+\ee_{j_1}+\dots+\ee_{j_i} }^G=\ket{u_{j_i}\dots u_{j_2}u_{j_1}\cdot \lambda}^G,\\
&\ket{\lambda-\ee_{j_1}-\dots-\ee_{j_i} }^G=\ket{V_{j_1}V_{j_2}\dots V_{j_i}\cdot \lambda}^G.
\end{align*}
Hence, we conclude
\begin{align*}
&E_i(a_{-1},a_{-2},\dots)\ket{\lambda}^G=\ket{e_i(\bm{u}_r)\cdot \lambda }^G+\ket{\Delta},\\
&E^{(r)}_i(a_{1},a_{2},\dots)\ket{\lambda}^G=\ket{e_i(\bm{V}_r)\cdot \lambda}^G,
\end{align*}
which lead the following proposition.
\begin{prop}\label{prop:action_on_G}
Let $\lambda$ be a partition and $r\geq \ell(\lambda)$.
For a symmetric polynomial $f(x_1,\dots,x_r)$, let $f(P_1,P_2,\dots)$ denote the polynomial in $P_1,P_2,\dots$ that satisfies $f(p_1(x),p_2(x),\dots,)=f(x)$.
We further define
\[
f^{(r)}(P_1,P_2,\dots):=f(\overbrace{-\beta,\dots,-\beta}^r,P_1,P_2,\dots).
\]
Then we have
\begin{align*}
&f(a_{-1},a_{-2},\dots)\ket{\lambda}^G=\ket{f(\bm{u}_r)\cdot \lambda }^G+\ket{\Delta},\quad \mbox{with}\quad \bra{0}e^{H(x)}\ket{\Delta}\in I_r,\\
&f^{(r)}(a_{1},a_{2},\dots)\ket{\lambda}^G=\ket{f(\bm{V}_r)\cdot \lambda}^G.
\end{align*}
\end{prop}
\begin{proof}
It suffices to prove $\bra{0}e^{H(x)}\ket{\Delta}\in I_r$.
Recall the fact that $\ket{\Delta}$ is a sum of vectors of the form \eqref{eq:remainder_term}.
Since they are contained in $I_r$ (Lemma \ref{lemma:evaluate1} \eqref{item:evaluate1_3}), we conclude $\bra{0}e^{H(x)}\ket{\Delta}\in I_r$.
\end{proof}

\subsection{$s_{\lambda/\mu}(\bm{v})$ and $s_{\lambda/\mu}(\bm{U})$}

Define the linear operators $v_i:\mathfrak{X}\to \mathfrak{X}$ and $U_i:\mathfrak{X}\to \mathfrak{X}$ that act on a basis $\lambda$ as
\begin{align*}
&v_i\cdot \lambda=
\begin{cases}
\lambda\cup \{\mbox{a box in $i$-th row} \}, & \mbox{if possible},\\
-\beta\cdot \lambda, & \mbox{otherwise}.
\end{cases}\\
&U_i\cdot \mu=
\begin{cases}
(-\beta)^{\zet{\mu-\ee_i}-\zet{\underline{\mu-\ee_{i}}}}\cdot \underline{\mu-\ee_i}, & \mu_i>0,\\
0, & \mu_i=0.
\end{cases}
\end{align*}
By seeing the action on the basis, we find that the operators $v_i$ satisfy the Knuth relation:
\begin{align*}
&v_iv_kv_j=v_kv_iv_j,\quad i\leq j<k,\\
&v_jv_iv_k=v_jv_kv_i,\quad i<j\leq k,
\end{align*}
and the $U_i$ satisfy the ``reverse'' Knuth relation:
\begin{align*}
&U_iU_kU_j=U_kU_iU_j,\quad i\geq j>k,\\
&U_jU_iU_k=U_jU_kU_i,\quad i>j\geq k.
\end{align*}
We can define the non-symmetric skew Schur polynomials $s_{\lambda/\mu}(\bm{v}_r)$ and $s_{\lambda/\mu}(\bm{U}_r)$ in the same way as Definition \ref{defi:noncommutativeSchur}.

\subsection{Actions on $\ket{n}^g$}

Let $n=(n_1,\dots,n_r)\in \ZZ^r$ and $i>0$.
By using \eqref{eq:relation_added} and the commutative relations $[a_{-i},e^{-\theta} ]=-(-\beta)^ie^{-\theta}$, $[a_{i},e^{-\theta}]=0$, we obtain
\begin{align}
&\{ a_{-i}+(r-1)(-\beta)^i\}\ket{n}^g=\sum_{j=1}^r\ket{n+i\ee_j }^g+\ket{\delta}\label{eq:a_-i_on_g},\\
&\hspace{4em} \mbox{where}\quad \ket{\delta}=\psi_{n_1-1}e^{-\theta}\dots  e^{-\theta}\psi_{n_r-r} a_{-i}\ket{-r},\nonumber\\
&a_{i}\ket{n}^g=\sum_{j=1}^r\ket{n-i\ee_j }^g\label{eq:a_i_on_g}.
\end{align}
\begin{rem}
Note that $\ket{\delta}$ is \textit{not} equal to 
\[
\psi_{n_1-1}e^{-\theta}\dots  e^{-\theta}\psi_{n_r-r}
\mbox{\fbox{$e^{-\theta}$}} a_{-i}\ket{-r}.
\]
See Remarks \ref{rem:n^g_tips} and \ref{rem:note}.
This is why the coefficient of $(-\beta)^i$ in \eqref{eq:a_-i_on_g} is not $r$ but $r-1$.
\end{rem}

From (\ref{eq:a_-i_on_g}--\ref{eq:a_i_on_g}), we have
\begin{align*}
&E_i^{(r-1)}(a_{-1},a_{-2},\dots)\ket{n}^g=\sum_{1\leq j_1< \dots< j_i\leq r}\ket{n+\ee_{j_1}+\dots+\ee_{j_i} }^g+\ket{\Delta_i},\\
&E_i(a_{1},a_{2},\dots)\ket{n}^g=\sum_{1\leq j_1< \dots< j_i\leq r}\ket{n-\ee_{j_1}-\dots-\ee_{j_i} }^g,
\end{align*}
where $\ket{\Delta_i}$ is a sum of the vectors
\begin{equation}\label{eq:remainder_term2}
a_{-j_1}\dots a_{-j_k}\cdot \psi_{n_1-1}e^{-\theta}\dots e^{-\theta}\psi_{n_r-r} a_{-s}\ket{-r}
\end{equation}
with $0\leq s<i$ and $j_1,j_2,\dots,j_k>0$ ($k\geq 0$).

If $n=\lambda$ is a partition, the $\ket{\lambda+\ee_{j_1}+\dots+\ee_{j_i} }^g$ and $\ket{\lambda-\ee_{j_1}-\dots-\ee_{j_i} }^g$ satisfy the assumption of Corollary \ref{cor:X(n)}.
Therefore, they can be written as 
\begin{align*}
&\ket{\lambda+\ee_{j_1}+\dots+\ee_{j_i} }^g
=(-\beta)^{\zet{\lambda+\ee_{j_1}+\dots+\ee_{j_i}}-\zet{\underline{\lambda+\ee_{j_1}+\dots+\ee_{j_i}}}}\ket{\underline{\lambda+\ee_{j_1}+\dots+\ee_{j_i}} }^g,\\
&\ket{\lambda-\ee_{j_1}-\dots-\ee_{j_i} }^g
=(-\beta)^{\zet{\lambda-\ee_{j_1}-\dots-\ee_{j_i}}-\zet{\underline{\lambda-\ee_{j_1}-\dots-\ee_{j_i}}}}\ket{\underline{\lambda-\ee_{j_1}-\dots-\ee_{j_i}} }^g.
\end{align*}
Further, by seeing the actions of $v_i$ and $U_i$, these equations are rewritten as
\begin{align*}
&\ket{\lambda+\ee_{j_1}+\dots+\ee_{j_i} }^g=\ket{v_{j_i}\dots v_{j_2}v_{j_1}\cdot \lambda}^g,\\
&\ket{\lambda-\ee_{j_1}-\dots-\ee_{j_i} }^g=\ket{U_{j_1}U_{j_2}\dots U_{j_i}\cdot \lambda}^g.
\end{align*}
Then we conclude
\begin{align*}
&E_i^{(r-1)}(a_{-1},a_{-2},\dots)\ket{\lambda}^g=\ket{e_i(\bm{v}_r)\cdot \lambda }^g+\ket{\Delta_i},\\
&E_i(a_{1},a_{2},\dots)\ket{\lambda}^g=\ket{e_i(\bm{U}_r)\cdot \lambda}^g.
\end{align*}
\begin{prop}\label{prop:action_on_g}
Let $\lambda$ be a partition and $r\geq \ell(\lambda)$.
For a symmetric polynomial $f(x_1,\dots,x_r)$, let $f(P_1,P_2,\dots)$ and $f^{(r-1)}(P_1,P_2,\dots)$ be as in Proposition \ref{prop:action_on_G}.
Then we have
\begin{align*}
&f^{(r-1)}(a_{-1},a_{-2},\dots)\ket{\lambda}^g=\ket{f(\bm{v}_r)\cdot \lambda }^g+\ket{\Delta_f},\\
&f(a_{1},a_{2},\dots)\ket{\lambda}^g=\ket{f(\bm{U}_r)\cdot \lambda}^g,
\end{align*}
where $\ket{\Delta_f}$ is a ``remainder term," which satisfies $\ket{\Delta_{s_\nu}}=0$ for $\ell(\lambda)+\ell(\nu)\leq r$.
\end{prop}
\begin{proof}
If suffices to prove $\ket{\Delta_{s_\nu}}=0$ for $\ell(\lambda)+\ell(\nu)\leq r$.
Let $f=s_{\nu}$.
Since $\ket{\Delta_f}$ is a sum of vectors of the form \eqref{eq:remainder_term2}, we have $\ket{\Delta_f}=0$ if $\lambda_{r}=\lambda_{r-1}=\dots=\lambda_{r-\ell(\nu)}=0$ (Lemma \ref{lemma:evaluate2} \eqref{item:evaluate2_2}). 
This is equivalent to $\ell(\lambda)+\ell(\nu)\leq r$.
\end{proof}

\section{Expansions of $s_\nu G_{\lambda \slsl \mu}$ and $s_\nu G_{\lambda/\mu}$ in stable Grothendieck polynomials}\label{sec:G-expan}

\subsection{The module of skew shapes}\label{sec:u-action}

By definition, a skew shape 
is identified with a pair of partitions $(\lambda,\mu)$ satisfying $\lambda\supset \mu$.
We extend the action of the linear operators $u_i,v_i,U_i,V_i$ to the $\QQ[\beta]$-module $\bigoplus_{\lambda\supset \mu}\QQ[\beta]\cdot (\lambda,\mu)$ of skew shapes by letting $u_i$ and $v_i$ act on $\lambda$ and $U_i$ and $V_i$ act on $\mu$:
\[
\begin{aligned}
&u_i\cdot (\lambda,\mu):=(u_i\cdot \lambda,\mu),\\
&U_i\cdot (\lambda,\mu):=(\lambda,U_i\cdot \mu),
\end{aligned}\qquad
\begin{aligned}
&v_i\cdot (\lambda,\mu):=(v_i\cdot \lambda,\mu),\\
&V_i\cdot (\lambda,\mu):=(\lambda,V_i\cdot \mu).
\end{aligned}
\]

\begin{example}
\def\Ygtab#1{{\tiny {\ygtab{#1}}}}%
\begin{align*}
\begin{array}{lll}
u_1\cdot \Ygtab{\bl & \\ \bl & \\ \\ }=\Ygtab{\bl & & \\ \bl & \\ \\ },\quad
&u_2\cdot \Ygtab{\bl & \\ \bl & \\ \\ }=-\beta\cdot \Ygtab{\bl & & \\ \bl & & \\ \\ },\quad
&u_3\cdot \Ygtab{\bl & \\ \bl & \\ \\ }=\Ygtab{\bl & \\ \bl & \\ & \\ },\\[1em]
U_1\cdot \Ygtab{\bl & \\ \bl & \\ \\ }=-\beta\cdot\Ygtab{& \\ & \\ \\},\quad
&U_2\cdot \Ygtab{\bl & \\ \bl & \\ \\ }=\Ygtab{\bl & \\ & \\ \\},\quad
&U_3\cdot \Ygtab{\bl & \\ \bl & \\ \\ }=0.
\end{array}
\end{align*}
\end{example}

\begin{example}
\def\Ygtab#1{{\tiny {\ygtab{#1}}}}%
\begin{align*}
\begin{array}{lll}
v_1\cdot \Ygtab{\bl & \\ \bl & \\ \\ }=\Ygtab{\bl & & \\ \bl & \\ \\ },\quad
&v_2\cdot \Ygtab{\bl & \\ \bl & \\ \\ }=-\beta\cdot \Ygtab{\bl & \\ \bl  & \\ \\ },\quad
&v_3\cdot \Ygtab{\bl & \\ \bl & \\ \\ }=\Ygtab{\bl & \\ \bl & \\ & \\ },\\[1em]
V_1\cdot \Ygtab{\bl & \\ \bl & \\ \\ }=-\beta\cdot\Ygtab{\bl & \\ \bl& \\ \\},\quad
&V_2\cdot \Ygtab{\bl & \\ \bl & \\ \\ }=\Ygtab{\bl & \\ & \\ \\},\quad
&V_3\cdot \Ygtab{\bl & \\ \bl & \\ \\ }=-\beta \Ygtab{\bl & \\ \bl & \\ \\ }.
\end{array}
\end{align*}
\end{example}

Obviously, as linear operators on $\bigoplus_{\lambda\supset \mu}\QQ[\beta]\cdot (\lambda,\mu)$, the $u_i,v_i,U_j,V_j$ satisfy the commutative relation: 
\begin{equation}\label{eq:comm1}
U_iu_j=u_jU_i,\quad
V_iv_j=v_jV_i.
\end{equation}

\subsection{Supersymmetric Schur function}\label{sec:superSchur}

Let $x=(x_1,x_2,\dots)$ and $y=(y_1,y_2,\dots)$ be two sets of indeterminate.
The \textit{supersymmetric Schur function} $s_\lambda(x/y)$ \cite[I, \S 5, Example 23]{macdonald1998symmetric} (see also \cite[\S 6.1, Exesise 8]{fulton_1996}) is the rational function that associates with the \textit{elementary supersymmetric functions} $e_0(x/y),e_1(x/y),\dots$ defined by
\[
\sum_{n=0}^{\infty}e_n(x/y)t^n:=\prod_{i}(1+x_it)\prod_{j}(1+y_jt)^{-1}.
\]
This is equivalent to saying
\begin{equation}\label{eq:super-elementary-function}
e_i(x/y)=h_0(y)e_i(x)-h_1(y)e_{i-1}(x)+\dots+(-1)^ih_i(y)e_0(x).
\end{equation}
In \cite[I, \S 5, Excesice 23, (1)]{macdonald1998symmetric}, Macdonald shows the following useful formula:
\begin{equation}\label{eq:super-Schur-expansion}
s_\lambda(x/y)=\sum_{\mu\subset\lambda} (-1)^{\zet{\lambda/\mu}}s_{(\lambda/\mu)'}(y)s_\mu(x),
\end{equation}
where $(\lambda/\mu)'$ is the transpose of $\lambda/\mu$.

For any symmetric function $f(x)$, we define the supersymmetric function $f(x/y)$ by seeing the expansion of $f$ in Schur polynomials.
For example, the $i$-th supersymmetric power sum $p_i(x/y)$ is expressed as
\begin{equation}\label{eq:super-powersum}
p_i(x/y)=(x_1^i+x_2^i+\cdots)-(y_1^i+y_2^i+\cdots).
\end{equation}

\subsection{Non-commutative supersymmetric Schur function}

\begin{defi}\label{defi:noncommutative_super_Schur}
We define the non-commutative supersymmetric Schur functions
$s_\lambda(\bm{u}_m/\bm{U}_n)$,
$s_\lambda(\bm{v}_m/\bm{V}_n)$
by
\begin{gather*}
s_\lambda(\bm{u}_m/\bm{U}_n):=\sum_{
\mu\subset \lambda
}
(-1)^{\zet{\lambda/\mu}}
s_{(\lambda/\mu)'}(\bm{U}_n)s_\mu(\bm{u}_m),\\
s_\lambda(\bm{v}_m/\bm{V}_n):=\sum_{
\mu\subset \lambda
}
(-1)^{\zet{\lambda/\mu}}
s_{(\lambda/\mu)'}(\bm{V}_n)s_\mu(\bm{v}_m).
\end{gather*}
\end{defi}

\subsection{$G_{\lambda\slsl\mu}$-expansion of $s_\nu G_{\lambda\slsl \mu}$}

Let $G_{\slsl}:\bigoplus_{\lambda\supset \mu}\QQ[\beta]\cdot (\lambda,\mu)\to \widehat{\Lambda}$ be the $\QQ[\beta]$-linear map that sends a skew diagram $(\lambda,\mu)$ to $G_{\lambda\slsl \mu}(x)$.
The following proposition provides an algorithm to express the product $s_\nu(x) G_{\lambda\slsl \mu}(x)$ as a linear combination of skew stable Grothendieck polynomials.
\begin{prop}\label{prop:product_sG//}
For $r\geq \ell(\lambda)$ and $s\geq \ell(\mu)$, the equation
\begin{gather*}
s_\nu(x)G_{\lambda\slsl \mu}(x)\equiv 
G_{\slsl}(s_\nu(\bm{u}_r/\bm{U}_s)\cdot (\lambda,\mu))
\mod{I_{r-s}}
\end{gather*}
holds.
In other words, we have
\[
s_\nu(x_1,\dots,x_{n})G_{\lambda\slsl \mu}(x_1,\dots,x_{n})
=
G_{\slsl}(s_\nu(\bm{u}_r/\bm{U}_s)\cdot (\lambda,\mu))\vert_{x_{n+1}=x_{n+2}=\cdots=0}
\]
for $n\leq r-s$.
\end{prop}
\begin{proof}
By Theorem \ref{prop:noncommutativeSchur}, we know that it suffices to prove
\[
p_i(x)G_{\lambda\slsl\mu}(x)\equiv 
G_{\slsl}(p_i(\bm{u}_r/\bm{U}_s)\cdot (\lambda,\mu) )
\mod{I_{r-s}}
\]
for $i=0,1,\dots,\ell(\nu)$.
Since
\[
p_i(x)e^{H(x)}=e^{H(x)}a_{-i}-a_{-i}e^{H(x)}, 
\]
we obtain
\begin{align*}
&p_i(x)G_{\lambda\slsl \mu}(x)\\
&=
{}^g\bra{\mu}e^{H(x)}a_{-i}\ket{\lambda}^G
-
{}^g\bra{\mu}a_{-i}e^{H(x)}\ket{\lambda}^G\\
&={}^g\bra{\mu}e^{H(x)}\ket{P_i(\bm{u}_r)\cdot \lambda}^G
+
{}^g\bra{-P_i(\bm{U}_s)\cdot \mu}e^{H(x)}\ket{\lambda}^G+{}^g\bra{\mu}e^{H(x)}\ket{\Delta}^G
\end{align*}
from Propositions  \ref{prop:action_on_G} and \ref{prop:action_on_g}.
The last term is a ``remainder term,'' which would not give a nice symmetric function.
However, from Lemma \ref{lemma:adjoint_distance} and Proposition \ref{prop:action_on_G}, we find
\[
{}^g\bra{\mu}e^{H(x)}\ket{\Delta}^G\in I_{r-s}.
\]
From \eqref{eq:super-powersum}, it follows that
\begin{align*}
p_i(x)G_{\lambda\slsl \mu}(x)
&
\equiv
{}^g\bra{\mu}e^{H(x)}\ket{P_i(\bm{u}_r)\cdot \lambda}^G
+
{}^g\bra{-P_i(\bm{U}_s)\cdot \mu}e^{H(x)}\ket{\lambda}^G\mod{I_{r-s}}
\\
&=
G_{\slsl}(p_i(\bm{u}_r/\bm{U}_s)\cdot (\lambda,\mu) ).
\end{align*}
\end{proof}

\subsection{$G_{\lambda/\mu}$-expansion of $s_\nu G_{\lambda/ \mu}$}

Let $G:\bigoplus_{\lambda\supset \mu}\QQ[\beta]\cdot (\lambda,\mu)\to \widehat{\Lambda}$ be the $\QQ[\beta]$-linear map defined by $(\lambda,\mu)\mapsto G_{\lambda/\mu}(x)$.

\begin{prop}\label{prop:product_sG/}
For $r\geq \ell(\lambda)$ and $s\geq \ell(\mu)$, the equation
\begin{gather*}
s_\nu(x)G_{\lambda/ \mu}(x)\equiv 
G(s_\nu(\bm{u}_r/\bm{U}_s)\cdot (\lambda,\mu))
\mod{I_{r-s}}
\end{gather*}
holds.
\end{prop}
\begin{proof}
The proof is exactly same as that of Proposition \ref{prop:product_sG//}.
\end{proof}

\begin{example}
Let $r=4$, $s=1$,
$\nu=\Ygtab{ \\}$, 
$\lambda=\Ygtab{ & & \\ \\ }$,
and $\mu=\Ygtab{ \\}$.
Since
\begin{align*}
s_{(1)}(\bm{u}_4/\bm{U}_1)
=h_0(\bm{U}_1)e_1(\bm{u}_4)-h_1(\bm{U}_1)e_0(\bm{u}_4)=u_1+u_2+u_3+u_4-U_1,
\end{align*}
we obtain $(s_\lambda=s_\lambda(x_1,x_2,x_3)$, $G_{\lambda\slsl\mu}=G_{\lambda\slsl\mu}(x_1,x_2,x_3))$
\begin{gather*}
s_{(1)}
G_{(3,1)\slsl(1)}
=
G_{(4,1)\slsl(1)}
+
G_{(3,2)\slsl(1)}
+
G_{(3,1,1)\slsl(1)}
-\beta
G_{(3,1,1,1)\slsl(1)}
-
G_{(3,1)},\\
s_{(1)}
G_{(3,1)/(1)}
=
G_{(4,1)/(1)}
+
G_{(3,2)/(1)}
+
G_{(3,1,1)/(1)}
-\beta
G_{(3,1,1,1)/(1)}
-
G_{(3,1)}.
\end{gather*}
\end{example}

\section{$g_{\lambda/\mu}$-expansion of $s_\nu g_{\lambda/\mu}$}\label{sec:g-expan}

Let $g:\bigoplus_{\lambda\supset \mu}\QQ[\beta]\cdot (\lambda,\mu)\to \Lambda$ be the $\QQ[\beta]$-linear map defined by $(\lambda,\mu)\mapsto g_{\lambda/\mu}(x)$.
\begin{prop}\label{prop:product_sg}
For $r\geq \ell(\lambda)+\ell(\nu)$, we have 
\[
s_\nu(x)g_{\lambda/\mu}(x)= g(s_\nu(\bm{v}_{r+1}/\bm{V}_r)\cdot (\lambda,\mu)).
\]
\end{prop}
\begin{proof}
The proof is similar to Proposition \ref{prop:product_sG//}.
It suffices to prove
\[
p_i(x)g_{\lambda/\mu}(x)=
g(p_i(\bm{v}_{r+1}/\bm{V}_r)\cdot (\lambda,\mu) )
\]
for $i=0,1,\dots,\ell(\nu)$.
Since 
\begin{align*}
p_i(x)e^{H(x)}
&=e^{H(x)}a_{-i}-a_{-i}e^{H(x)}\\
&=e^{H(x)}\{a_{-i}+r(-\beta)^i\}-\{a_{-i}+r(-\beta)^i\}e^{H(x)},
\end{align*}
we have
\begin{align*}
&p_i(x)g_{\lambda/\mu}(x)\\
&=
{}^G\bra{\mu}e^{H(x)}P_i^{(r)}(a_{-1},a_{-2},\dots)\ket{\lambda}^g
-
{}^G\bra{\mu}P_i^{(r)}(a_{-1},a_{-2},\dots)e^{H(x)}\ket{\lambda}^g\\
&=
{}^G\bra{\mu}e^{H(x)}\ket{P_i(\bm{v}_{r+1})\cdot \lambda}^g
+
{}^G\bra{-P_i(\bm{V}_r)\cdot \mu}e^{H(x)}\ket{\lambda}^g\\
&=g(p_i(\bm{v}_{r+1}/\bm{V}_r)\cdot (\lambda,\mu))
\end{align*}
from Propositions  \ref{prop:action_on_G} and \ref{prop:action_on_g}.
\end{proof}

\begin{example}
Let $r=1$,
$\nu=(n)$, and
$\lambda=\mu=\emptyset$.
Then we have $s_{(n)}(\bm{v}_1/\bm{V}_0)=h_n(v_1)=v_1^n$, which leads $s_{(n)}(x)\cdot g_\emptyset(x)=g(v_1^n\cdot \emptyset)=g_{(n)}(x).$
\end{example}

\begin{example}
Let $r=n$,
$\nu=(1^n)$, and
$\lambda=\mu=\emptyset$.
Then 
\[
s_{(1^n)}(\bm{v}_{n+1}/\bm{V}_n)
=
\sum_{j=0}^n(-1)^je_{n-j}(\bm{v}_{n+1})h_j(\bm{V}_n).
\]
Since
$
e_{i}(\bm{v}_{n+1})\cdot \emptyset=\sum_{k=0}^{i} (-\beta)^{i-k}
{n-k \choose i-k}
\cdot (1^k)$ and
$
h_j(\bm{V}_n)\cdot \emptyset=
\beta^j
{-n \choose j}\cdot \emptyset
$,
we have 
\begin{align*}
s_{(1^n)}(x)\cdot g_\emptyset(x)
&=g\left(\sum_{k=0}^n\sum_{j=0}^n(-\beta)^{n-k}{n-k \choose n-j-k}{-n \choose j}\cdot (1^k)\right)\\
&=g\left(\sum_{k=0}^n(-\beta)^{n-k}{-k\choose n-k}\cdot (1^k)\right)
=\sum_{k=0}^{n}\beta^{n-k}{n-1\choose n-k}g_{(1^k)}(x).
\end{align*}
%
\end{example}

\begin{example}
Let $r=3$,
$\nu=\Ygtab{ & \\}$, and
$\lambda/\mu=\Ygtab{\bl & \\ \\}$.
Since
\[
s_{\nu}(\bm{v}_{4}/\bm{V}_3)
=
e_0(\bm{V}_3)h_2(\bm{v}_{4})
-
e_1(\bm{V}_3)h_1(\bm{v}_{4})
+
e_2(\bm{V}_3)h_0(\bm{v}_{4}),
\]
we have
\bgroup
\def\Ygtab#1{{\tiny {\ygtab{#1}}}}%
\begin{align*}
&\begin{aligned}
h_2(\bm{v}_4)\cdot \Ygtab{\bl & \\ \\}
&=
\Ygtab{\bl & & & \\ \\}
+
\Ygtab{\bl & & \\ & \\}
+
\Ygtab{\bl & & \\ \\ \\}
-\beta\cdot
\Ygtab{\bl & & \\ \\}
-2\beta\cdot
\Ygtab{\bl & \\ &  \\}\\
&\hspace{3em}+
\Ygtab{\bl & \\ & \\ \\}
-2\beta\cdot
\Ygtab{\bl & \\ \\ \\ }
+\beta^2\cdot
\Ygtab{\bl & \\ \\}
\end{aligned}\\
&\begin{aligned}
e_1(\bm{V}_3)h_1(\bm{v}_{4})\cdot \Ygtab{\bl & \\ \\}
&=e_1(\bm{V}_3)\left(
\Ygtab{\bl & & \\ \\}
+
\Ygtab{\bl & \\ & \\}
+
\Ygtab{\bl & \\ \\ \\}
-\beta\cdot
\Ygtab{\bl & \\ \\}
\right)\\
&=
\Ygtab{ & & \\ \\}
+
\Ygtab{ & \\ & \\}
+
\Ygtab{ & \\ \\ \\}
-\beta\cdot
\Ygtab{ & \\ \\}\\
&\hspace{1em}
-2\beta\cdot
\Ygtab{\bl & & \\ \\}
-2\beta\cdot
\Ygtab{\bl & \\ & \\}
-2\beta\cdot
\Ygtab{\bl & \\ \\ \\}
+2\beta^2\cdot
\Ygtab{\bl & \\ \\}
\end{aligned}\\
&\begin{aligned}
e_2(\bm{V}_3)\cdot \Ygtab{\bl & \\ \\}
=
-2\beta\cdot
\Ygtab{ & \\ \\}
+\beta^2
\Ygtab{\bl & \\ \\}
\end{aligned}.
\end{align*}
\egroup
From them, we conclude
\begin{align*}
s_{\tiny \Ygtab{ & \\}}g_{\tiny \Ygtab{ \bl & \\ \\}}
&=
g_{\tiny \Ygtab{\bl & & & \\ \\}}
+
g_{\tiny \Ygtab{\bl & & \\ & \\}}
+
g_{\tiny \Ygtab{\bl & & \\ \\ \\}}
+
g_{\tiny \Ygtab{\bl & \\ & \\ \\}}\\
&\hspace{2em}
-
g_{\tiny \Ygtab{ & & \\ \\}}
-
g_{\tiny \Ygtab{ & \\ & \\}}
-
g_{\tiny \Ygtab{ & \\ \\ \\}}
+\beta
g_{\tiny \Ygtab{\bl & & \\ \\}}
-\beta
g_{\tiny \Ygtab{ & \\ \\}}.
\end{align*}
\end{example}

\begin{example}
It is interesting to consider the case $\lambda=\mu$.
In this case, it follows that $g_{\lambda/\mu}(x)=g_\emptyset(x)=1$.

For example, if $\lambda=\mu=(1)$, $\nu=(2)$, and $r=2$, we have
\begin{align*}
s_{(2)}(x)g_\emptyset(x)
&=
g\left(
\sum_{i=0}^{2}e_i(\bm{V}_2)h_{2-i}(\bm{v}_3)\cdot ((1),(1))
\right)\\
&=
g\big(
((3),(1))
+
((2,1),(1))
-
((2),\emptyset)
-
((1,1),\emptyset)
-\beta
((1,1),(1))
\big)\\
&
=g_{(2,1)/(1)}(x)-g_{(1,1)}(x)-\beta g_{(1)}(x).
\end{align*}
Note that $g_{(3)/(1)}(x)=g_{(2)}(x)$.

Replacing $\lambda$ with various partitions, we can obtain a number of nontrivial relations between dual stable Grothendieck polynomials.
For example:
\begin{align*}
s_{(2)}(x)
&=g_{(2)}(x) & (\lambda=\mu=\emptyset)\\
&=g_{(2,1)/(1)}(x)-g_{(1,1)}(x)-\beta g_{(1)}(x).& (\lambda=\mu=(1))
\end{align*}

\begin{align*}
&s_{(3)}(x)\\
&=g_{(3)}(x) & (\lambda=\mu=\emptyset)\\
&
=
g_{(3,1)/(1)}(x)
-g_{(2,1)}(x)
-\beta g_{(2,1)/(1)}(x)
+\beta g_{(1,1)}(x)
+\beta^2 g_{(1)}(x)
& (\lambda=\mu=(1))\\
&=
g_{(3)}(x)
+g_{(2,1)}(x)
+ g_{(1,1,1)}(x)
-g_{(2,1,1)/(1)}(x)
+\beta g_{(1,1)}(x)
& (\lambda=\mu=(1,1))\\
&=
g_{(3,2)/(2)}(x)
-
g_{(2,2)/(1)}(x)
-\beta
g_{(2)}(x).
& (\lambda=\mu=(2))
\end{align*}
\end{example}

\section{Expansion of $s_\nu^\perp G_{\lambda/\mu}$ and $s_\nu^\perp g_{\lambda/\mu}$}\label{sec:expansion_of_s^perp}

\subsection{Transposed action}
We now introduce a new ``transposed'' action of $u_i,v_i,U_i,V_i$ on $(\lambda,\mu)$ by letting $u_i$ and $v_i$ act on $\mu$ and $U_i$ and $V_i$ act on $\lambda$:
\[
\begin{aligned}
&u_i\ast (\lambda,\mu):=
\begin{cases}
(\lambda,u_i\cdot \mu), & \lambda\supset u_i\cdot \mu\\
0, & \mbox{otherwise}
\end{cases}\\
&U_i\ast (\lambda,\mu):=
\begin{cases}
(U_i\cdot \lambda,\mu), & U_i\cdot \lambda \supset \mu\\
0, & \mbox{otherwise}
\end{cases}
\end{aligned}\qquad
\begin{aligned}
&v_i\ast (\lambda,\mu):=
\begin{cases}
(\lambda,v_i\cdot \mu), &\lambda \supset v_i\cdot \mu\\
0,& \mbox{otherwise}
\end{cases}\\
&V_i\ast (\lambda,\mu):=
\begin{cases}
(v_i\cdot \lambda,\mu), &v_i\cdot \lambda \supset \mu\\
0, & \mbox{otherwise}.
\end{cases}
\end{aligned}
\]

\begin{prop}
For $r\geq \ell(\lambda)$,
we have
\begin{enumerate}
\item 
$
s_\nu^\perp G_{\lambda/ \mu}(x)=
G
\left(
s_\nu(\bm{V}_r/(-\beta)^r )\ast (\lambda,\mu)
\right)
$,
\item 
$
s_\nu^\perp g_{\lambda/ \mu}(x)=g
\left(
s_\nu(\bm{U}_r)\ast (\lambda,\mu)
\right)
$.
\end{enumerate}
\end{prop}
\begin{proof}
From Proposition \ref{prop:action_on_G}, it follows that
\[
\{a_i+r(-\beta)^i\}\ket{\lambda}^G=P^{(r)}_i(a_1,a_2,\dots)\ket{\lambda}^G
=\ket{P_i(\bm{V}_r)\cdot \lambda}^G,
\]
which is equivalent to 
\[
a_i\ket{\lambda}^G=\ket{P_i(\bm{V}_r)\cdot \lambda}^G-r(-\beta)^r\ket{\lambda}^G
=\ket{P_i(\bm{V}_r/(-\beta)^r)\cdot \lambda}^G.
\]
By using Corollary \ref{cor:adjoint_action}, we have
\begin{align*}
p_i^\perp G_{\lambda/ \mu}(x)
&=
{}^g\bra{\mu}e^{-\Theta}e^{H(x)}a_i\ket{\lambda}^G
=
{}^g\bra{\mu}e^{-\Theta}e^{H(x)}\ket{P_i(\bm{V}_r/(-\beta)^r)\cdot\lambda}^G\\
&=
G
\left(
P_i(\bm{V}_r/(-\beta)^r )\ast (\lambda,\mu)
\right),
\end{align*}
which implies (i).

(ii) is also proved by the same calculations as (i).
\end{proof}

\begin{example}
Let $\nu=
\ygtab{
 & \\
\\
}$ and $r=3$.
Young tableaux of shape $\nu$ with entries $1'>2'>3'$ are listed as follows:
\[
\Ygtab{
2' & 1'\\
1'\\
}\ 
\Ygtab{
2' & 2'\\
1'\\
}\ 
\Ygtab{
3' & 1'\\
1'\\
}\ 
\Ygtab{
3' & 2'\\
1'\\
}\ 
\Ygtab{
3' & 3'\\
1'\\
}\ 
\Ygtab{
3' & 1'\\
2'\\
}\ 
\Ygtab{
3' & 2'\\
2'\\
}\ 
\Ygtab{
3' & 3'\\
2'\\
}
\]
Therefore, we have
\begin{align*}
s_\nu(\bm{U}_3)
&=
U_1U_2U_1+U_1U_2U_2
+U_1U_3U_1+U_1U_3U_2\\
&\hspace{3em}+U_1U_3U_3+U_2U_3U_1
+U_2U_3U_2+U_2U_3U_3.
\end{align*}
For example, 
\begin{align*}
&
s^\perp_{\tiny
\Ygtab{
 & \\
\\
}
}
g_{\tiny
\Ygtab{
 & \\
\\
}
}
=1,\\
& 
s^\perp_{\tiny
\Ygtab{
 & \\
\\
}
}
g_{\tiny
\Ygtab{
 & \\
 & \\
}
}
=
g_{\tiny
\Ygtab{
\\
}
}
-\beta
,
\\
&
s^\perp_{\tiny
\Ygtab{
 & \\
\\
}
}
g_{\tiny
\Ygtab{ \bl & \bl & \\
& & \\
\\}
}
=
g_{\tiny
\Ygtab{
\bl\bullet \\
\\
\\
}
}
+
g_{\tiny
\Ygtab{
\bl\bullet & \bl\bullet\\
 & \\
}
}
-\beta
g_{\tiny
\Ygtab{
\bl\bullet\\
 \\
}
}
+
g_{\tiny
\Ygtab{
\bl\bullet& \bl\bullet & \\
\bl\bullet \\
\\
}
}
=
g_{\tiny
\Ygtab{
\\
\\
}
}
+
g_{\tiny
\Ygtab{
& \\
}
}
+
g_{\tiny
\Ygtab{
\bl & \\
\\
}
}
-\beta
g_{\tiny
\Ygtab{
\\
}
}.
\end{align*}
\end{example}

\section*{Acknowledgments}

This work is partially supported by JSPS Kakenhi Grant Number 19K03605.
The author is grateful to Professor Takeshi Ikeda for his comments on the manuscript and for his suggestions for future research.

\providecommand{\bysame}{\leavevmode\hbox to3em{\hrulefill}\thinspace}
\providecommand{\MR}{\relax\ifhmode\unskip\space\fi MR }
\providecommand{\MRhref}[2]{%
  \href{http://www.ams.org/mathscinet-getitem?mr=#1}{#2}
}
\providecommand{\href}[2]{#2}


\end{document}